\documentclass[11pt]{amsart}
\usepackage{a4wide}
\usepackage{version}
\usepackage{amsmath}
\usepackage{amsfonts, amssymb, amsthm}
\usepackage{pstricks}
\usepackage{color,mathrsfs,epsfig}
\usepackage{epsfig}
\usepackage{color}

\DeclareMathAlphabet{\mathpzc}{OT1}{pzc}{m}{it}
\newtheorem{thm}{Theorem}
\newtheorem{cor}[thm]{Corollary}
\newtheorem{preremark}[thm]{Remark}

\numberwithin{equation}{section}
\newtheorem{prop}[thm]{Proposition}
\newtheorem{lemma}[thm]{Lemma}
\newtheorem{defn}[thm]{Definition}

\numberwithin{equation}{section}

\newcommand{\norm}[1]{\left\Vert#1\right\Vert}
\newcommand{\abs}[1]{\left\vert#1\right\vert}

\newcommand{\R}{\mathbb R}
\newcommand{\N}{\mathbb N}

\newcommand{\dd} {\; \mathrm{d}}

\DeclareMathOperator*{\osc}{osc}

\newcommand{\meanbar}[1]{%
\setbox0 = \hbox{$#1 \int$}
\hbox to 0pt{%
\thinspace
\hskip 0.1\wd0
\raise 0.5\ht0
\hbox{%
\lower 0.5\dp0
\hbox{\rule{0.8\wd0}{2\linethickness}}
}%
\hss
}%
}

\makeatletter
\def\@tocline#1#2#3#4#5#6#7{\relax
  \ifnum #1>\c@tocdepth 
  \else
    \par \addpenalty\@secpenalty\addvspace{#2}%
    \begingroup \hyphenpenalty\@M
    \@ifempty{#4}{%
      \@tempdima\csname r@tocindent\number#1\endcsname\relax
    }{%
      \@tempdima#4\relax
    }%
    \parindent\z@ \leftskip#3\relax \advance\leftskip\@tempdima\relax
    \rightskip\@pnumwidth plus4em \parfillskip-\@pnumwidth
    #5\leavevmode\hskip-\@tempdima
      \ifcase #1
       \or\or \hskip 1em \or \hskip 2em \else \hskip 3em \fi%
      #6\nobreak\relax
    \hfill\hbox to\@pnumwidth{\@tocpagenum{#7}}\par
    \nobreak
    \endgroup
  \fi}
\makeatother

\author{Georgiana Chatzigeorgiou and Emmanouil Milakis}

\title[Regularity for fully nonlinear parabolic equations with oblique boundary data]{Regularity for fully nonlinear parabolic equations with oblique boundary data}

\begin{document}

\begin{abstract} 
We obtain up to a flat boundary regularity results in parabolic H\"{o}lder spaces for viscosity solutions of fully nonlinear parabolic equations with oblique boundary conditions.
\end{abstract}

\let\thefootnote\relax\footnotetext{


This work was co-funded by the European Regional Development Fund and the Republic of Cyprus through the Research and Innovation Foundation (Project: EXCELLENCE/1216/0025)
}




\maketitle

\tableofcontents

\section{Introduction}

The purpose of the present article is to study the regularity of viscosity solutions of fully nonlinear parabolic equations with oblique boundary conditions of the form
\begin{align}  \label{main_problem}
\begin{cases}
F(D^2u) -u_t = f, &\ \ \ \ \ \text{ in } \ Q_1^+ \\
\beta \cdot Du = g, &\ \ \ \ \ \text{ on } \ Q_1^*\\
u = u_0, &\ \ \ \ \ \text{ on } \ \partial_pQ_1^+ \setminus Q_1^*
\end{cases}
\end{align}
where $F$ is a uniformly elliptic convex operator in $S_n$, $f, g$ and $u_0$ are given data and $\beta:Q_1^* \to \R^n$ is a given vector function with $\beta_n \geq \delta_0>0$ and $||\beta||_{L^\infty} \leq 1$. By $Q_1^+$ we denote the half parabolic cylinder with flat part $Q_1^*$ (see subsection \ref{notations} for precise definitions). 

There is a vast literature that concerns oblique derivative boundary value problems for elliptic operators. For the linear elliptic case we refer the reader to the book of Gary Lieberman \cite{LiebObl} and references therein. In the case of fully nonlinear elliptic operators, existence and uniqueness of viscosity solutions are obtained in \cite{Ishii91} (where boundary conditions are in fact more general). Regularity of viscosity solutions can be found in \cite{ms1} and \cite{LZ}.

The corresponding theory for linear parabolic equations  with oblique derivative boundary data is also well understood. For existence, uniqueness and regularity results we refer to \cite{Lieb86I}, \cite{Lieb86II}, \cite{NU95}, \cite{Naz90}, \cite{Weid05} and \cite{GS84}. For the case when the operator is fully nonlinear parabolic, comparison and existence results for viscosity solutions can be found in \cite{IshiiSato04}. Interior and boundary estimates for fully nonlinear parabolic equations with Dirichlet conditions have been studied by Lihe Wang in a series of papers (see \cite{Wang1}, \cite{Wang2}, \cite{Wang3}). Moreover apriori H\"older estimates for classical solutions appeared in \cite{Ural94}, \cite{Lieb90}. The main goal of the present paper is to investigate the regularity of viscosity solutions.

Our purpose is to prove, under suitable assumptions, H\"older regularity (in the parabolic sense) for $u$ and its first and second derivatives (note that in the definition of viscosity solutions we only assume $u$ to be continuous). The idea is to use an approximation method as used (for the elliptic case) in \cite{LZ} which is first introduced in \cite{CaffNonlinear} (see also \cite{CC}). That is, we try to approximate inductively the general problem (\ref{main_problem}) by "simpler" ones for which the regularity is known. The "simpler" problem will be special case of (\ref{main_problem}) where the equation as well as the boundary condition are homogeneous and the vector $\beta$ is constant. To attack the regularity for this type of problems we first examine the regularity for the parabolic Neumann problem (that is, when $\beta=(,0,\dots,0,1)$)  which is obtained by adapting the ideas of \cite{ms1} in the parabolic framework. Then, we observe that after a suitable change of variables a constant oblique derivative problem can de transformed into a Neumann problem.

The outline of the paper is as follows. Section \ref{prelims} contains the basic notations and definition as well as an estimate of Aleksandrov-Bakel'man-Pucci-Tso type which is a basic tool in our approach. In Section \ref{Ca}, we prove H\"older estimates for $u$ via a boundary Harnack inequality. In Section \ref{C1a} we introduce suitable approximate solutions to get a uniqueness type result which is necessary when we study the first order difference quotients. Next we get H\"older estimates for the first derivatives for the Neumann and oblique derivative case respectively. In Section \ref{C2a} we obtain H\"older estimates for the second derivatives first for the Neumann and secondly for oblique derivative case. Finally, in the appendix, for the sake of completeness, we provide proofs for certain regularity results for the nonlinear parabolic Dirichlet problem and a closedness result which are used in the text.   

\section{Preliminaries}\label{prelims}

\subsection{Notations-Definitions}\label{notations}
We denote $X=(x,y) \in \R^n$, where $x \in \R^{n-1}$ and $y \in \R$ and $P=(X,t)\in \R^{n+1}$, where $X$ are the space variables and $t$ is the time variable. The Euclidean ball in $\R^n$ and the elementary cylinder in $\R^{n+1}$ will be denoted by
$$B_r(X_0) := \{ X \in \R^n : |X-X_0| < r\}, \ \ Q_r(X_0,t_0) := B_r (X_0) \times (t_0 - r^2, t_0]$$
respectively. We define the following half and thin-cylinders, for $r>0, X_0 \in \R^n_+, t_0 \in \R$
$$Q_r^+(X_0,t_0):= Q_r(X_0,t_0) \cap \{y >y_0 \},  \ \ Q_r^*(X_0,t_0):= Q_r(X_0,t_0) \cap \{y =y_0 \}.$$
Note that, $\Omega^\circ, \overline{\Omega}, \partial \Omega$ will be the interior, the closure and the boundary of the domain $\Omega \subset \R^{n+1}$, respectively, in the sense of the Euclidean topology of $\R^{n+1}$. We define also the parabolic interior to be,
$$int_p(\Omega ) := \{ (X,t) \in \R ^{n+1} : \ there \ exists \  r>0 \ so \ that \ Q_r^\circ(X,t) \subset \Omega \}$$
and the parabolic boundary, $\partial _p (\Omega ) := \overline{\Omega } \setminus int_p(\Omega ).$ Let us also define the parabolic distance for $P_1=(X,t),\ P_2=(Y,s) \in \R^{n+1} $, $p(P_1,P_2) := \max \{|X-Y|,|t-s|^{1/2}\}$. Note that in this case $Q_r(P_0)$ will be the set $\{ P \in \R^{n+1} : p(P,P_0)<r, t<t_0 \}$.

Next we define also the corresponding parabolic H\"older spaces. For a function $f$ defined in a domain $\Omega \subset \R^{n+1}$ we set,

$$[f]_{\alpha; \Omega} := \sup _ {P_1, P_2 \in \overline{\Omega}, P_1 \neq P_2} \frac{|f(P_1) - f(P_2)|}{p(P_1,P_2)^{\alpha}}, \ \ \langle f \rangle _{\alpha+1; \Omega} := \sup _ {(X,t_1), (X,t_2) \in \overline{\Omega} \atop t_1 \neq t_2} \frac{|f(X,t_1) - f(X,t_2)|}{|t_1-t_2|^{\frac{\alpha+1}{2}}}.$$

Then we say that,
\begin{enumerate}
\item[$\bullet$] $f \in H^\alpha (\overline{\Omega})$ if
$||f||_{H^{\alpha} (\overline{\Omega}) } := \sup _{\overline{\Omega}} |f| + [f]_{\alpha; \Omega}  < + \infty .$
\item[$\bullet$] $f \in H^{\alpha+1} (\overline{\Omega})$ if 
$$||f||_{H^{\alpha+1} (\overline{\Omega}) } := \sup _{\overline{\Omega}} |f| +  \sum_{i=1}^n  \sup _{\overline{\Omega}} |D_if| + \sum_{i=1}^n [D_if]_{\alpha; \Omega} +\langle f \rangle _{\alpha+1; \Omega}  < + \infty.$$
\item[$\bullet$] $f \in H^{\alpha+2} (\overline{\Omega})$ if
\begin{align*}
||f||_{H^{\alpha+2} (\overline{\Omega}) } :=& \sup _{\overline{\Omega}} |f| +  \sum_{i=1}^n  \sup _{\overline{\Omega}} |D_if| +\sup _{\overline{\Omega}} |f_t|+ \sum_{i,j=1}^n  \sup _{\overline{\Omega}} |D_{ij}^2f|\\
&+[f_t]_{\alpha; \Omega}
+ \sum_{i,j=1}^n [D_{ij}^2f]_{\alpha; \Omega} +\sum_{i=1}^n \langle D_i f \rangle _{\alpha+1; \Omega}  < + \infty .
\end{align*}
\end{enumerate}
Due to the nonlinear character of our problem, we will mainly prove $H^{\alpha+1}$ and $H^{\alpha+2}$-regularity results in the punctual sense at a point. We say that $u$ is punctually $H^{\alpha+1}$ at a point $P_1 \in \overline{\Omega}$ if there exists $R_{1;P_1}(X)= A_{P_1}+ B_{P_1} \cdot (X-X_1)$, where $A_{P_1} \in \R$ and $B_{P_1} \in \R^{n}$ and some cylinder $\overline{Q}_{r_1}(P_1) \subset \Omega$, so that for any $0<r<r_1$,
$$|u(X,t) - R_{1;P_1}(X)| \leq K \ r^{1+\alpha}, \ \ \text{ for every } \ (X,t) \in \overline{Q}_{r}(P_1)$$
for some constant $K>0$. We say that  $u$ is punctually $H^{\alpha+2}$ at a point $P_1 \in \overline{\Omega}$ if the above holds when we replace $R_{1;P_1}(X)$ by $R_{2;P_1}(X,t)= A_{P_1}+ B_{P_1} \cdot (X-X_1) + C_{P_1} (t-t_1) + \frac{1}{2} (X-X_1)^\tau D_{P_1} (X-X_1)$, where $A_{P_1}, C_{P_1} \in \R, B_{P_1} \in \R^{n}$ and $D_{P_1} \in \R_{n \times n}$ and estimating by $r^{2+\alpha}$ instead of $r^{1+\alpha}$.
Note that when we study points on a flat part of the boundary, cylinders in the above definitions are replaced by half-cylinders. 

The nonlinear operator $F$ is uniformly elliptic which means that there exist constants $0<\lambda \leq \Lambda $ such that
\begin{align} \label{ell_cond}
\lambda ||N||_\infty \leq F(M+N,X,t) - F(M,X,t) \leq \Lambda ||N||_\infty
\end{align}
for every $M, N \in S_n$ with $N \geq 0$ and $(X,t) \in \Omega$, where we denote by $S_n$ the space of symmetric $n \times n$ matrices with real entries.

\begin{defn} \label{visc_sol1}
We say that a continuous $u$ is a viscosity subsolution (supersolution) of $F(D^2u)-u_t=f$ in $\Omega$ if, whenever a smooth test function $\phi $ touches $u$ by above (below) at some point $(X_0,t_0) \in \Omega$ we have that
\begin{equation}\label{visc_subsol}
F(D^2 \phi (X_0,t_0),X_0,t_0) - \phi_t(X_0,t_0) \geq \ (\leq ) \ f(X_0,t_0).
\end{equation}
Recall that we say that $v$ touches $u$ by above (below) at a point $(X_0,t_0)$ if $u(X_0,t_0)=v(X_0,t_0)$ and $u \leq \ (\geq) \ v$ in some cylinder $Q_r(X_0,t_0)$. We say that $u$ is a viscosity solution of $F(D^2u)-u_t=f$ in $\Omega$ if it is both a viscosity subsolution and supersolution.
\end{defn}

\begin{defn} \label{Sclass}
We say that $u \in \underline{S}_p (\lambda,\Lambda,f)$ in $\Omega$, if it is a viscosity subsolution of 
$$\mathcal{M}^+ (D^2u(X,t),\lambda,\Lambda)-u_t(X,t)=f(X,t)$$ 
in $\Omega$ and that $u \in \overline{S}_p (\lambda,\Lambda,f)$ in $\Omega$, if it is a viscosity supersolution of 
$$\mathcal{M}^- (D^2u(X,t),\lambda,\Lambda)-u_t(X,t)=f(X,t)$$
in $\Omega$, where $\mathcal{M}^\pm$ denote the Pucci's extremal operators. In addition we define, $S_p (\lambda,\Lambda,f) := \underline{S}_p (\lambda,\Lambda,f) \cap \overline{S}_p (\lambda,\Lambda,f).$
\end{defn}

\begin{defn} \label{oblique_cond}
We say that $\beta \cdot Du \geq \ (\leq ) \ g$, on $Q^*_r$ in the viscosity sense if whenever we take any point $P_0=(x_0,0,t_0) \in Q^*_r$ and a smooth test function $\phi$ that touches $u$ by above (below) at $P_0$ in some half-cylider $Q^+_\rho(P_0) \subset Q^+_r$ then we must have that $\beta (x_0,t_0)\cdot D\phi  (P_0) \geq \ (\leq ) \ g(x_0,t_0).$ If both hold at the same time we say that $\beta \cdot Du = g$ on $Q^*_r$ in the viscosity sense.
\end{defn}

Note that a special case of the oblique-type condition is when $\beta(x,t)=(0,\dots,0,1) \in \R^n$ for every $(x,t) \in Q_r^*$ which is the \textit{Neumann boundary condition}.

\begin{preremark} \label{curved_bdary}
Due to the local character of our approach, in what follows we will always assume that $u$ equals to $u_0$ on $\partial_pQ_1^+ \setminus Q_1^*$.
\end{preremark}

We call a constant $C>0$ universal if it depends only on $n, \lambda, \Lambda, \delta_0$ and other constants related to function $\beta$.


\subsection{An Aleksadrov-Bakel'man-Pucci-Tso type estimate}
We prove an ABPT-type maximum principle corresponding to our oblique derivative problem (see \cite{LZ}, \cite{ms1} for the elliptic case).

Recall that the convex envelope of a function $u \in C \left( \overline{Q}_1^+ \right)$ is defined as
\begin{align*}
\Gamma(u)(X,t)&:=\sup \{v(X,t) : v(Z,s) \leq u(Z,s)  \ \text{ in } \ Q_1^+ \ \text{ and } \ v \text{ is convex w.r.t. }  Z \text{ and decreasing w.r.t. } s \} \\
&=\sup \{ \xi \cdot X+h : \xi \cdot Z +h \leq u(Z,s), \text{ for every } \ Z \in B_1^+, s \in (-1,t] \}.
\end{align*}
Moreover for smooth enough $v$ we define the function
$$G(v)(X,t)=(Dv(X,t),v(X,t)-X \cdot Dv(X,t)).$$
Note that $\det D_{(X,t)} G(v)=v_t \ \det D^2v$.

\begin{thm} \label{ABP-Oblique}
(ABPT-estimate in the case of Oblique boundary data). Let $f \in C \left( \overline{Q}_r^+ \right), g \in C \left( \overline{Q}_r^* \right)$ and $u \in \overline{S}_p(\lambda, \Lambda,f) \cap C\left( \overline{Q}_r^+\right)$ with $\beta \cdot Du \leq g$ on $Q_r^*$ in the viscosity sense. Then,
\begin{equation} \label{ABP-Oblique_est}
\sup_{Q_r^+} u^- \leq \sup_{\partial_pQ^+_r \setminus Q_r^*} u^- +C r \left( \int_{\{u =
\Gamma_u\}} \abs{f^+(X,t)}^{n+1} \dd X \dd t \right)^{1/n+1} + C r \sup_{Q_r^*}g^+
\end{equation}
where $\Gamma_u$ is the convex envelope of $-u^-:=\min\{u,0\}$ in $Q_r^+$ and $C>0$ is universal constant.
\end{thm}

\begin{proof}
For convenience take $r=1$ and $\sup_{\partial_pQ^+_1 \setminus Q_1^*} u^-=0$. We denote by $M:= \sup_{Q_1^+}u^->0$ then there exists $(X_0,t_0) \in Q_1^+ \cup Q_1^*$ (since $u \geq 0$ on $\partial_pQ^+_1 \setminus Q_1^*$) so that $u^-(X_0,t_0)=M$.

Note that if $\sup_{Q_1^*}g^+ \geq \frac{\delta M}{16}$ then (\ref{ABP-Oblique_est}) holds. So we consider the case when $\sup_{Q_1^*}g^+ < \frac{\delta M}{16}$.

Since $\Gamma_u \in H^2(\overline{Q}_1^+)$ then we can show (see \cite{Wang1} or \cite{ImbertSil} (for more details) and references therein), using area formula 
$$|G(\Gamma_u)(Q_1^+)| \leq \int_{Q_1^+ \cap \{u=\Gamma_u\}} -(\Gamma_u)_t\det (D^2\Gamma_u)\ dXdt$$
and $-(\Gamma_u)_t+\lambda \Delta (\Gamma_u) \leq f^+$, in $\{u=\Gamma_u\}$. Thus we get $\ |G(\Gamma_u)(Q_1^+)| \leq \int_{Q_1^+ \cap \{u=\Gamma_u\}} (f^+)^{n+1}\ dXdt$.

We consider the set
$$\mathcal{A}:=\{(\xi,h) \in \R^n\times \R: |\xi|<\frac{M}{8}<\frac{M}{2}\leq -h \leq \frac{3M}{4}, \xi_n \geq \frac{M}{8}, |\xi'|\leq \frac{\delta M}{16} \}$$
where $\xi':=(\xi_1,\dots,\xi_{n-1})$. We will show that $\mathcal{A} \subset G(\Gamma_u)(Q_1^+)$. Take any $(\xi,h) \in \mathcal{A}$ and we consider $P(X):= \xi \cdot X+h.$ Then we observe that for every $X \in \overline{B}_2$, $P(X)\leq |\xi||X|+h\leq -\frac{M}{4}<0$. In addition one has $P(X_0)-u(X_0.t_0)\geq -|\xi||X_0|+h+M\geq \frac{M}{8}>0$, that is $\max_{\overline{B}_1^+} (P(X)-u(X,t_0))\geq 0.$

Define
$$t_1:=\sup \lbrace -1 \leq t \leq 0 : \text{ for every } s \in [-1,t], \ \max_{\overline{B}_1^+} (P(X)-u(X,s))<0 \rbrace.$$
Note that $t_1 \leq t_0 \leq 0$ and from the continuity of $P-u$ with respect to $s$ we have that 
$$P(X_1)-u(X_1,t_1)=\max_{\overline{B}_1^+} (P(X)-u(X,t_1))=0.$$
This shows that $(X_1,t_1) \in Q_1^+$. Indeed if $(X_1,t_1) \in \partial_pQ^+_1 \setminus Q_1^*$ then we would have $u(X_1,t_1) \geq 0$ and since $P(X_1)<0$, $P(X_1)-u(X_1,t_1)<0$ we get a contradiction. If now $(X_1,t_1) \in Q_1^*$, $P$ touches $u$ by below at $(X_1,t_1)$, then $\beta(x_1,t_1) \cdot \xi \leq g(x_1,t_1)$ but
\begin{align*}
\beta(x_1,t_1) \cdot \xi \geq \delta \xi_n - |\xi '|||\beta||_{L^\infty}\geq \delta \xi_n - \frac{\delta \xi_n}{2} > \sup_{Q_1^*} g^+
\end{align*}
since $\xi_n > \frac{M}{8}>\frac{2}{\delta} \sup_{Q_1^*} g^+$ and we get a contradiction.

Combining the above we have that $P(X) \leq -u^-(X,t)$, for every $X \in \overline{B}_1^+$, $-1<t\leq t_1$ and $P(X_1)=-u^-(X_1,t_1)$. Then $P(X)$ touches $\Gamma_u$ by below at $(X_1,t_1) \in Q_1^+$, thus $G(\Gamma_u)(X_1,t_1) = (\xi, h)$. Since $|\mathcal{A}| = C(\delta,n) M^{n+1}$ the proof is complete.
\end{proof}

\subsection{A useful change of variables}\label{changeofvar}
Here we consider the case when the function  $\beta$ is constant. In this case we see that using a suitable change of variables, a viscosity problem of the form
\begin{align} \label{oblique_constant}
\begin{cases}
F(D^2u) -u_t = 0, &\ \ \ \ \ \text{ in } \ Q_1^+ \\
\beta \cdot Du = 0, &\ \ \ \ \ \text{ on } \ Q_1^*\\
\end{cases}
\end{align}
can be transformed into a nonlinear Neumann parabolic problem 
\begin{align} \label{neumann}
\begin{cases}
\tilde{F}(D^2v) -v_t = 0, &\ \ \ \ \ \text{ in } \ \tilde{Q_1^+} \\
v_y = 0, &\ \ \ \ \ \text{ on } \ Q_1^*\\
\end{cases}
\end{align}
where $\tilde{F}$ is also an elliptic operator on $S_n$ and $\tilde{Q_1^+}$ a suitable "half-set".

\par More precisely, consider the transformation
\begin{align*}
A:=\left(  \begin{matrix}
 1 \dots 0 &  \frac{\beta_{1}}{\beta_n}\\
 \ddots \\
0 \dots 1 & \frac{\beta_{n-1}}{\beta_n} \\
0 \dots 0 & 1
\end{matrix} \right).
\end{align*}
For a smooth function $\psi = \psi(z,w,t)$ we define $\phi(x,y,t):= \psi \left( A(x,y),t \right)$ and we can easily check that $D^2 \phi = A^\tau D^2 \psi A$ and $D^2\psi = (A^{-1})^\tau D^2\phi A^{-1}.$

Define $\tilde{F}(M):= F((A^{-1})^\tau M A^{-1})$. Then $\tilde{F}$ is elliptic and its ellipticity constants are universal multiples of $\lambda$ and $\Lambda$. For, we use the fact that the norms $||A||_\infty$, $||A^{-1}||_\infty, ||A^\tau||_\infty$ and $||(A^{-1})^\tau||_\infty$ are bounded from above by $\frac{\delta_0+1}{\delta_0}=:C_{\delta_0}$ combined with the ellipticity of $F$. We only need to be careful in observing that for $M, N \in S_n$ with $N \geq 0$ then $(A^{-1})^\tau N A^{-1}$ is symmetric (easily checked by calculations) and positive definite. To get the positivity we observe that $\det \left((A^{-1})^\tau N A^{-1}\right) = \det(N)\geq 0$, since $\det A=1$ and $N$ is non-negative definite. Moreover, $\left((A^{-1})^\tau N A^{-1}\right)_{ij} = \sum_{l,k=1}^n N_{kl}b_{ki}b_{lj}=N_{ij}$, for $i <n,\ j<n$. Then Sylvester's criterion  gives that $(A^{-1})^\tau N A^{-1} \geq 0$.

\par We observe also that the transformation $A$ maps the hyper-plane $\{y=0\}$ identically into itself and the half-space $\{y>0\}$ into itself (so does $A^{-1}$). So,  $\tilde{Q^+_1}:= \{ (x,y,t)=(A^{-1}(z,w),t), \ \text{ for } \ (z,w,t) \in Q_1^+\}$ lies in the half-space $\{y>0\}$ and $Q_1^*$ is part of its parabolic boundary.

Note that combining all the above one can ensure that if $u(Z,t)$ is a viscosity solution of (\ref{oblique_constant}) then $v(X,t)=u(AX,t)$ is a viscosity solution of (\ref{neumann}). This fact will be useful later to prove regularity for problems of the form (\ref{oblique_constant}) using the regularity of problems of the form (\ref{neumann}).

\section{H\"older Estimates}\label{Ca}

In the present section we prove H\"older regularity up to the flat part of the boundary at which we assume a viscosity oblique derivative condition proving first a boundary Harnack-type inequality.

\begin{thm}(Up to the flat boundary $H^\alpha$-regularity). \label{bdary_C^a_oblique}
Let $f$ and $g$ be continuous and bounded in $Q_1^+$ and $Q_1^*$ respectively. Assume that $u \in C\left( Q_1^+ \cup Q_1^*\right) $ is such that
$$\begin{cases} 
u \in S_p(\lambda, \Lambda,f), &\ \ \ \ \text{ in } \ \ Q_1^+ \\
\beta \cdot Du =g,&\ \ \ \ \text{ on } \ \ Q_1^*, \ \text{ in the viscosity sense}.
\end{cases}$$
Then for universal constants $C>0$ and $0<\alpha <1$, we have that $u \in H^\alpha \left( \overline{Q}^+_{1/2} \right)$, with an estimate 
\begin{equation} \label{C^a_ineq2_oblique}
||u||_{H^\alpha \left( \overline{Q}^+_{1/2} \right) } \leq C \left( ||u||_{L^\infty \left( Q_1^+ \right)} + ||f||_{L^{n+1}\left( Q_1^+ \right)}+||g||_{L^\infty \left( Q_1^* \right)} \right).
\end{equation}
\end{thm}

Combining the interior Harnack inequality with a barrier argument we get the following boundary Harnack inequality (see \cite{LZ}, \cite{LiebTrud} for the elliptic case).

\begin{thm}(Boundary Harnack inequality). \label{bdary_Harnack_oblique}
Let $f$ and $g$ be continuous and bounded in $Q_1^+$ and $Q_1^*$ respectively. Assume that $u \in \left( Q_1^+ \cup Q_1^*\right)$, $u \geq 0$ is such that
$$\begin{cases} 
u \in S_p(\lambda, \Lambda,f), &\ \ \ \ \text{ in } \ \ Q_1^+ \\
\beta \cdot Du =g,&\ \ \ \ \text{ on } \ \ Q_1^*, \ \text{ in the viscosity sense}.
\end{cases}$$
Then for universal constants $C>0$ and $0<\rho <1$, we have
\begin{equation} \label{bd_Harnack}
\sup_{K_{\frac{rR}{2}}(A,0)} u \leq C \left( \inf_{H \left(\frac{r}{4},\rho \right)} u + r^{\frac{n}{n+1}}||f||_{L^{n+1}\left( Q_1^+ \right)}+r||g||_{L^\infty \left( Q_1^* \right)} \right)
\end{equation}
for every $0<r<\frac{1}{2}$, where $A=(0,\dots,0,r) \in \R^n$, $K_R := B_{\frac{R^2}{2 \sqrt{2} }}(0,0) \times \left[ -R^2 + \frac{3}{8} R^4, -R^2+ \frac{4}{8} R^4 \right]$, for some universal $0<R <<1$ and 
$$H(r,\rho):= \{ (X,t) : |x| < \frac{rR^2}{4}, 0<y<\rho r, -\frac{r^2R^4}{16}<t\leq 0\}.$$
\end{thm}

\begin{proof}
For $0<r<\frac{1}{2}$ note that $Q_{r/2}(A,0) \subset \{ (X,t) : |x| < r, \frac{ r}{2}<y<\frac{3 r}{2}, -r^2<t\leq 0\}.$
Then we can apply interior Harnack inequality to $u$ in $Q_{r/2}(A,0)$ (see Theorem 2.4.32 in Section 2.4.3 of \cite{ImbertSil}),
$$\sup_{K_{\frac{r R}{2}}(A,0)} u \leq C \left( \inf_{Q_{\frac{r R^2}{2}}(A,0)} u + r ^{\frac{n}{n+1}}||f||_{L^{n+1}(Q_1)} \right).$$
Let
$$H'(r,\rho):= \{ (X,t) : |x| < \frac{rR^2}{4}, y=\rho r, -\frac{r^2R^4}{16}<t\leq 0\}.$$
Note that if we choose $0<\rho < \frac{\sqrt{3}R^2}{4}$ then $H'(r,\rho ) \subset Q_{\frac{r R^2}{2}}(A,0)$. So we want to show that
\begin{equation}\label{bd_Harnack_1}
B:=\inf_{H '\left(r,\rho \right)} u \leq C \left( \inf_{H \left(\frac{r}{4},\rho \right)} u + r^{\frac{n}{n+1}}||f||_{L^{n+1}\left( Q_1^+ \right)}+r||g||_{L^\infty \left( Q_1^* \right)} \right).
\end{equation}
In other words we want to find a suitable lower bound for $u$ in $H \left(\frac{r}{4},\rho \right)$. We do this comparing $u$ with a suitable barrier function.

For $\bar{r}:=\frac{rR^2}{4}$ we define
$$b(X,t) := B -\frac{B}{4} \left[ 2 - \frac{y^2}{(\rho r)^2}-\frac{y}{\rho r}+4 \left( \frac{|x|^2-t}{\bar{r}^2}\right) \right] - \frac{||g||_{L^\infty \left( Q_1^* \right)}}{\delta_0}(\rho r- y).$$

Then we compute in $H(r,\rho)$ 
$$\mathcal{M}^- (D^2b,\lambda,\Lambda)-b_t=\lambda \frac{B}{2(\rho r)^2}-(n-1)\Lambda \frac{2B}{\bar{r}^2}-\frac{B}{\bar{r}^2}\geq 0.$$ 
by choosing $0<\rho \leq \sqrt{\frac{\lambda R^4}{32\left[2(n-1)\Lambda +1 \right]}}\ $. Hence we have $u-b \in \overline{S}_p (\lambda, \Lambda,f)$, in $H(r,\rho).$

Next, we study $b$ on the parabolic boundary of $H(r,\rho)$. 
On $\overline{H}(r,\rho) \cap \{y=0\}$ we have that
\begin{align*}
\beta \cdot Db &= -\frac{2B}{\bar{r}^2} \beta \cdot (x,0) + \frac{B}{4\rho r}\beta_n+\frac{||g||_{L^\infty \left( Q_1^* \right)}}{\delta_0} \beta_n \geq \frac{B}{r} \left( -\frac{8}{R^2}+\frac{\delta_0}{4 \rho} \right)+||g||_{L^\infty \left( Q_1^* \right)} \\
&\geq ||g||_{L^\infty \left( Q_1^* \right)}, \ \text{ choosing } \ 0<\rho \leq \frac{\delta_0R^2}{32}.
\end{align*} 
On $\{|x|=\bar{r}\}$ we have that
\begin{align*}
b(X,t)&=B -\frac{B}{4} \left( 1 - \frac{y^2}{(\rho r)^2}\right) -\frac{B}{4} \left(1-\frac{y}{\rho r}\right)-B+B\frac{t}{\bar{r}^2} - \frac{||g||_{L^\infty \left( Q_1^* \right)}}{\delta_0}(\rho r- y)\\
&\leq 0 \leq u(X,t).
\end{align*}
The case $\{t=-\bar{r}^2\}$ is treated similarly. Finally on $\{y=\rho r\}$ we have that $b(x, \rho r,t)=B-B\left( \frac{|x|^2-t}{\bar{r}^2}\right) \leq B \leq u(x, \rho r,t).$ Hence, $\beta \cdot D(u-b) \leq 0$, on $\overline{H}(r,\rho)\cap \{y=0\}$ and $u-b \geq 0$, on $\partial_p H(r,\rho) \setminus \overline{H}(r,\rho)\cap \{y=0\}.$

Therefore from Theorem \ref{ABP-Oblique} we have that $u-b \geq -r^{\frac{n}{n+1}}||f||_{L^{n+1}\left( Q_1^+ \right)}$, in $H(r,\rho)$. Then in $H\left(\frac{r}{4},\rho \right)$ we have
\begin{align*}
u +Cr ||g||_{L^\infty \left( Q_1^* \right)} +r^{\frac{n}{n+1}}||f||_{L^{n+1}\left( Q_1^+ \right)} &\geq B -\frac{B}{2} \left[ 1+2\left(\frac{|x|^2-t}{\bar{r}^2}\right)\right] \\
&\geq B-\frac{B}{2}-\frac{B}{2}\cdot\frac{1}{8}=\frac{7B}{16}
\end{align*}
and the proof is complete.
\end{proof}

Theorem \ref{bdary_C^a_oblique} follows in a standard way.

\begin{proof}[Proof of Theorem \ref{bdary_C^a_oblique}]
For $0<r\leq \frac{1}{2}$ we consider quantities
$$M_r:=\sup_{Q_r^+}u, \ \ \ m_r:=\inf_{Q_r^+}u.$$
Then functions $v_1:=M_r-u$, $v_2:=u-m_r$ are non-negative in $Q_r^+$, $v_i \in S_p(\lambda,\Lambda,f)$ in $Q_r^+$ and $b\cdot Dv_i =g$ on $Q_r^*$. We apply Theorem \ref{bdary_Harnack_oblique} to $v_i$ and obtain
$$\osc_{Q_r^+}u \leq \osc_{Q_r^+}u+\osc_{K_{\frac{rR}{4}}(A,0)} u \leq C \left(\osc_{Q_r^+}u - \osc_{H \left(\frac{r}{8},\rho \right)} u + 2r^{\frac{n}{n+1}}||f||_{L^{n+1}\left( Q_r^+ \right)}+2r||g||_{L^\infty \left( Q_r^* \right)} \right)$$
thus
$$\osc_{Q_{\frac{\rho R^2}{2^5}r}^+} u \leq \gamma\osc_{Q_r^+}u +2C\left( r^{\frac{n}{n+1}}||f||_{L^{n+1}\left( Q_r^+ \right)}+r||g||_{L^\infty \left( Q_r^* \right)} \right)$$
where $\gamma := \frac{C-1}{C}<1$, since $Q_{\frac{\rho R^2}{2^5}r}^+ \subset H \left(\frac{r}{8},\rho \right)$. The result follows by a standard iteration argument.
\end{proof}

\section{H\"older Estimates for the first derivatives}\label{C1a}

In this section, we study existence and regularity of the first derivatives of viscosity solutions in the Neumann case (subsection 4.2) and then in the general oblique derivative case (subsection 4.3). To study the Neumann problem we define suitable difference quotients and apply the H\"older estimates proved in the previous section. To do so we have to explore which problem the difference of two solutions satisfies. This is achieved with the aid of suitable approximate solutions defined in subsection 4.1 (the idea had been initially introduced by Jensen for nonlinear elliptic equations). In subsection 4.3, first we use the change of variables of section \ref{changeofvar} and combining with the $H^{1+\alpha}$-estimates for Neumann problems of subsection \ref{H^1+a-neumann-section} we get $H^{1+\alpha}$-estimates for a constant oblique derivative problem. Secondly, we use a standard approximation method (see for example \cite{CC}, Chapter 8) and approximate a general oblique derivative problem by suitable constant oblique derivative problems.

\subsection{Approximate sub/super-solutions}
Let $u \in C \left( Q_1^+ \cup Q_1^* \right),\ \epsilon >0 $ and $0<\rho< \frac{1}{2}$. We define the sub-convolution of $u$ by
$$ u^{\epsilon, \rho} (X,t) = \sup_{(Z,s) \in \overline{Q}^+_\rho} \left( u(Z,s) - \frac{1}{\epsilon} |X-Z|^2 - \frac{1}{\epsilon}(t-s)^2 \right)$$
for any $(X,t) \in  Q_1^+ \cup Q_1^*$. The super-convolution $u_{\epsilon, \rho}$ is defined accordingly taking infimum and adding (instead of subtracting) the paraboloid.

Next we study some basic properties of $u^{\epsilon, \rho} (X,t)$ which will be useful in the sequel. An analog result holds for $u_{\epsilon,\rho}$ as well.

\begin{lemma} \label{propertiesofappr} $\ $
\begin{enumerate}
\item[\textbf{\textit{(i)}}] For $(X_0,t_0) \in Q_1^+ \cup Q_1^*$ there exists a point $(X_0^*,t_0^*) \in \overline{Q}_\rho^+$ so that
$$u^{\epsilon, \rho} (X_0,t_0) = u(X_0^*,t_0^*) - \frac{1}{\epsilon} |X_0-X_0^*|^2 - \frac{1}{\epsilon}(t_0-t_0^*)^2. $$
Moreover, $|X_0-X_0^*|^2 +(t_0-t_0^*)^2 \leq \epsilon \osc_{Q_\rho^+}u$, that is, as $\epsilon$ gets smaller $(X_0^*,t_0^*)$ gets closer to $(X_0,t_0)$.
\item[\textbf{\textit{(ii)}}] $ u^{\epsilon, \rho}$ is continuous in  $Q_1^+ \cup Q_1^*$.
\item[\textbf{\textit{(iii)}}] $u^{\epsilon, \rho} \to u$ uniformly in $\overline{Q}_\rho^+$, as $\epsilon \to 0^+$.
\item[\textbf{\textit{(vi)}}] $ \left( u^{\epsilon, \rho} \right) _y \geq 0$ on $Q_1^*$ in the viscosity sense. 
\end{enumerate}
\end{lemma}

\begin{proof} $\ $
\begin{enumerate}
\item[\textbf{\textit{(i)}}]The first part is immediate. For the second note that 
$$|X_0-X_0^*|^2 + (t_0-t_0^*)^2 = \epsilon \left( u(X_0^*,t_0^*) -u^{\epsilon, \rho} (X_0,t_0) \right)$$
and that $u^{\epsilon, \rho} (X_0,t_0) \geq u (X_0,t_0)$.
\item[\textbf{\textit{(ii)}}]  Take any $(X_1,t_1), (X_2,t_2) \in Q_1^+ \cup Q_1^*$, then for any $(Z,s) \in Q_\rho^+$ we have
\begin{align*}
u^{\epsilon, \rho} (X_1,t_1) &\geq  u(Z,s) - \frac{1}{\epsilon} |X_1-Z|^2 - \frac{1}{\epsilon}(t_1-s)^2\\
&\geq u(Z,s) - \frac{1}{\epsilon} |X_2-Z|^2 - \frac{1}{\epsilon} |X_1-X_2|^2 - \frac{2}{\epsilon} |X_2-Z||X_1-X_2|\\
&\ \ \ \ - \frac{1}{\epsilon}(t_2-s)^2 - \frac{1}{\epsilon}(t_1-t_2)^2  - \frac{2}{\epsilon}|t_2-s||t_1-t_2| \\
& \geq u(Z,s) - \frac{1}{\epsilon} |X_2-Z|^2 - \frac{1}{\epsilon}(t_2-s)^2 - \frac{6}{\epsilon} |X_1-X_2| -\frac{6}{\epsilon}|t_1-t_2|.
\end{align*}
Taking supremum over $Q_\rho^+$ we obtain $
|u^{\epsilon, \rho} (X_1,t_1)-u^{\epsilon, \rho} (X_2,t_2)| \leq \frac{6}{\epsilon} \left( |X_1-X_2| +|t_1-t_2| \right)$.
\item[\textbf{\textit{(iii)}}] Take any $M>0$. We know that $u$ is uniformly continuous in the compact set $\overline{Q}_\rho^+$, so there exists some $\delta(M)>0$ so that $|u(X,t)-u(Z,s)| < M$, for any $(X,t), (Z,s) \in \overline{Q}^+_\rho$ with $|X-Z|, |t-s| < \delta.$ We choose $0<\epsilon < \frac{\delta^2(M)}{\osc_{\overline{Q}_\rho^+}u}$ (note that if $\osc_{\overline{Q}_\rho^+}u=0$ then $u$ as well as $u^{\epsilon,\rho}$ are both identical zero and the result is obvious). Then taking any $(X_0,t_0) \in \overline{Q}_\rho^+$ we have that $|X_0-X_0^*|^2 + (t_0-t_0^*)^2 \leq \delta^2$. Therefore $|u(X_0^*,t_0^*)-u(X_0,t_0)| < M$ and we conclude that $0 \leq  u^{\epsilon, \rho} (X_0,t_0) - u (X_0,t_0) <M$.
\item[\textbf{\textit{(iv)}}] Let $\phi$ be a test function that touches $u^{\epsilon, \rho}$ by above at some point $(X_0,t_0) \in Q_1^*$. Let $(X_0^*,t_0^*) \in \overline{Q}_\rho^+$ be the point in \textbf{\textit{(i)}}. We have 
$$\phi(X,t) \geq u^{\epsilon, \rho}(X,t) \geq u(X_0^*,t_0^*) - \frac{1}{\epsilon} |X-X_0^*|^2 - \frac{1}{\epsilon}(t-t_0^*)^2 $$
in a half-cylinder around $(X_0,t_0)$. In particular $\phi(X_0,t_0)= u(X_0^*,t_0^*) - \frac{1}{\epsilon} |X_0-X_0^*|^2 - \frac{1}{\epsilon}(t_0-t_0^*)^2.$ Hence the function $\Phi(X) =\phi(X,t_0) -u(X_0^*,t_0^*)+\frac{1}{\epsilon} |X-X_0^*|^2 + \frac{1}{\epsilon}(t_0-t_0^*)^2 $ is non-negative near $X_0$ and zero at $X_0$. Therefore
$$ \Phi_y (X_0) = \lim_{h \to 0^+} \frac{\Phi(x_0,h) - \Phi(X_0)}{h} \geq 0.$$
That is, $\phi_y(X_0,t_0) - \frac{2}{\epsilon} y_0^* \geq 0$. But $y_0^* \geq 0$, thus we have that $\phi_y(X_0,t_0) \geq0$.
\end{enumerate}
\end{proof}

\begin{lemma} \label{starpoint}
Assume that $u$ is continuous in $Q_1^+ \cup Q_1^*$ and satisfies the condition $u_y \geq 0$ on $Q_1^*$ in the viscosity sense. Then for any $(X_0,t_0) \in Q_1^+$ the point $(X_0^*,t_0^*)$ of \textbf{\textit{(i)}} in Lemma \ref{propertiesofappr} lies in $\overline{Q}_\rho^+ \setminus Q_\rho^*$.
\end{lemma}

\begin{proof}
Take any $(X_0,t_0) \in Q_1^+$. We assume that $(X_0^*,t_0^*) \in Q_\rho^*$ to get a contradiction. Recall that $u^{\epsilon, \rho} (X_0,t_0) = u(X_0^*,t_0^*) - \frac{1}{\epsilon} |X_0-X_0^*|^2 - \frac{1}{\epsilon}(t_0-t_0^*)^2$
and that for any $(Z,s) \in \overline{Q}_\rho^+$,  $u^{\epsilon, \rho} (X_0,t_0) \geq u(Z,s) - \frac{1}{\epsilon} |X_0-Z|^2 - \frac{1}{\epsilon}(t_0-s)^2.$
That is for any $(Z,s) \in \overline{Q}_\rho^+$,
$$ u(X_0^*,t_0^*) - \frac{1}{\epsilon} |X_0-X_0^*|^2 - \frac{1}{\epsilon}(t_0-t_0^*)^2  \geq u(Z,s) - \frac{1}{\epsilon} |X_0-Z|^2 - \frac{1}{\epsilon}(t_0-s)^2.$$
Setting $\phi(Z,s) := u(X_0^*,t_0^*) - \frac{1}{\epsilon} |X_0-X_0^*|^2 - \frac{1}{\epsilon}(t_0-t_0^*)^2 +\frac{1}{\epsilon} |X_0-Z|^2 + \frac{1}{\epsilon}(t_0-s)^2$ we ensure that $ \phi \geq u$ in $\overline{Q}_\rho^+$ and $\phi(X_0^*,t_0^*)=u(X_0^*,t_0^*)$ which implies that $\phi_y(X_0^*,t_0^*) \geq 0$, But, on the other hand we can compute $\phi_y(X_0^*,t_0^*) = - \frac{2}{\epsilon}(y_0-y_0^*)=- \frac{2}{\epsilon}y_0 <0$. 
\end{proof}

\begin{lemma} \label{probforappr}
Let $u \in C \left( Q_1^+ \cup Q_1^* \right)$ satisfies in the viscosity sense
\begin{align} \label{probappr}
\begin{cases} 
F(D^2u)-u_t \geq 0, &\ \ \ \ \text{ in } \ \ Q_1^+ \\
u_y \geq 0,&\ \ \ \ \text{ on } \ \ Q_1^*.
\end{cases}
\end{align}
Then for any $0<\rho_1<\rho<\frac{1}{2}$ there exists some $0<\epsilon_0=\epsilon_0(\rho_1, \rho,u)$ such that for any $0<\epsilon<\epsilon_0$, $u^{\epsilon, \rho}$ is a viscosity subsolution of $F(D^2v)-v_t = 0$ in $Q_{\rho_1}^+$ (hence $u^{\epsilon, \rho}$ satisfies (\ref{probappr}) in $Q_{\rho_1}^+ \cup Q_{\rho_1}^*$).
\end{lemma}

Note that we do not use the Neumann condition of (\ref{probappr}) to show that $u^{\epsilon, \rho}$ satisfies the same condition since $u^{\epsilon, \rho}$ satisfies this condition anyway. However the Neumann condition is needed in order to get that $u^{\epsilon, \rho}$ is a subsolution of the equation (regarding Lemma \ref{starpoint}).

\begin{proof}
Take any point $(X_0,t_0) \in Q_{\rho_1}^+$ and any second order paraboloid $R_2(X,t)= A+ B \cdot (X-X_0) + C (t-t_0) + \frac{1}{2} (X-X_0)^\tau D (X-X_0)$ touching $u^{\epsilon, \rho}$ by above at $(X_0,t_0)$. We want to show that $F(D) -C \geq 0$.

Consider the translation 
$$\tilde{R}_2(X,t)=R_2(X+X_0-X_0^*, t+t_0-t_0^*)+\frac{1}{\epsilon} |X_0-X_0^*|^2 + \frac{1}{\epsilon}(t_0-t_0^*)^2.$$
Our aim is to show that for small $\epsilon$ this paraboloid touches $u$ at $(X_0^*,t_0^*)$ in order to apply the equation for $u$ (recall that $(X_0^*,t_0^*) \in \overline{Q}_\rho^+ \setminus Q_\rho^*$). Note that $\tilde{R}_2(X_0^*,t_0^*)= R_2(X_0, t_0)+\frac{1}{\epsilon} |X_0-X_0^*|^2 + \frac{1}{\epsilon}(t_0-t_0^*)^2 =u(X_0^*,t_0^*)$. Hence it remains to show that $\tilde{R}_2$ stays above $u$ around $(X_0^*,t_0^*)$.

Let $d=\rho - \rho_1>0$ and take $\epsilon_0 = \frac{d^4}{16 \osc_{\overline{Q}_\rho^+}u}>0$. Then, for $0<\epsilon\leq \epsilon_0$ we have that $|X_0-X_0^*|^2 + (t_0-t_0^*)^2 \leq \left( \frac{d}{2} \right)^4$ which ensures that $(X_0^*,t_0^*)$ is an interior point of $Q_\rho^+$. Therefore, we may choose some small enough $\delta>0$ so that $Q_\delta (X_0^*,t_0^*) \subset Q_\rho^+$ and $Q_\delta (X_0,t_0) \subset Q_\rho^+$. Note that if $(X,t) \in Q_\delta (X_0^*,t_0^*)$, then $(X+X_0-X_0^*, t+t_0-t_0^*) \in Q_\delta (X_0,t_0)$. Hence,  
$$u^{\epsilon, \rho} (X+X_0-X_0^*, t+t_0-t_0^*) \geq u(Z,s)-\frac{1}{\epsilon} |Z-X-X_0+X_0^*|^2 - \frac{1}{\epsilon}(s-t-t_0+t_0^*)^2$$
for any $(Z,s) \in \overline{Q}_\rho^+$. Taking $(Z,s)=(X,t)$,
$$R_2(X+X_0-X_0^*, t+t_0-t_0^*) \geq u^{\epsilon, \rho} (X+X_0-X_0^*, t+t_0-t_0^*) \geq u(X,t)-\frac{1}{\epsilon} |X_0-X_0^*|^2 - \frac{1}{\epsilon}(t_0-t_0^*)^2.$$
That is $u(X,t) \leq \tilde{R}_2(X,t)$, for $(X,t) \in Q_\delta (X_0^*,t_0^*)$ as desired.
\end{proof}

\begin{prop} \label{diff_Neumann}
Assume that $u, v \in C \left( Q_1^+ \cup Q_1^* \right)$ satisfy in the viscosity sense
\begin{align} \label{uniq1}
\begin{cases}
F(D^2u) -u_t \geq 0, &\ \ \ \ \ \text{ in } \ Q_1^+ \\
u_y \geq 0, &\ \ \ \ \ \text{ on } \ Q_1^* 
\end{cases} \ \ \text{ and } \ \
\begin{cases} 
F(D^2v) -v_t \leq 0, &\ \ \ \ \ \text{ in } \ Q_1^+ \\
v_y \leq 0, &\ \ \ \ \ \text{ on } \ Q_1^* 
\end{cases}
\end{align}
Then
\begin{align} \label{uniq2}
\begin{cases}
u-v \in \underline{S}_p \left( \frac{\lambda}{n}, \Lambda \right), &\ \ \ \ \ \text{ in } \ Q_1^+ \\
(u-v)_y \geq 0, &\ \ \ \ \ \text{ on } \ Q_1^* \ \  \text{ (in the viscosity sense). }
\end{cases}
\end{align} 
\end{prop}

\begin{proof}
In Theorem 4.6 of \cite{Wang2}, L.Wang uses a similar approximate consideration to obtain that $u-v \in \underline{S}_p \left( \frac{\lambda}{n}, \Lambda \right)$ in $Q_1^+ $. Hence it remains to examine the Neumann condition. 

We define the corresponding approximate sub/super-solutions $u^{\epsilon, \rho}, v_{\epsilon, \rho}$, for which we have that $\left( u^{\epsilon, \rho}-  v_{\epsilon, \rho} \right)_y \geq 0$ on $Q_1^*$ in the viscosity sense. This can be proved using the same idea as in the proof of \textbf{\textit{(iv)}}, Lemma \ref{propertiesofappr}. We are aiming to pass to the limit using Proposition \ref{neumannclosedness} (see appendix). To do so we take any $(X_0,t_0) \in Q_1^*$ and consider $0<\rho_0<\rho<1$ be so that $(X_0,t_0) \in Q_{\rho_0}^* \subset \overline{Q}_{\rho_0}^+ \subset \overline{Q}_{\rho}^+$. Lemma \ref{probforappr} gives that for sufficiently small $\epsilon>0$, $u^{\epsilon, \rho}, v_{\epsilon, \rho}$ are sub/super-solutions of $F(D^2w) -w_t = 0$ in $Q_{\rho_0}^+$. So again from Theorem 4.6 of \cite{Wang2}, $u^{\epsilon, \rho} - v_{\epsilon, \rho} \in \underline{S}_p \left( \frac{\lambda}{n}, \Lambda \right)$ in $Q_{\rho_0}^+$.

We now apply Proposition \ref{neumannclosedness} to $u^{\epsilon, \rho}-  v_{\epsilon, \rho}$ and combining with \textbf{\textit{(iii)}} of Lemma \ref{propertiesofappr} we obtain that $(u-v)_y \geq 0$ on $Q_{\rho_0}^*$ in the viscosity sense.
\end{proof}

Note that the above together with Theorem \ref{ABP-Oblique} gives a uniqueness result for the nonlinear Neumann problem.

\subsection{$H^{1+\alpha}$-estimates for the homogeneous Neumann case} \label{H^1+a-neumann-section}

First note that interior estimates for the first derivatives are proved in Section 4.2. of \cite{Wang2}. Actually, as explained in \cite{Wang2}, we have more than typical spatial $H^{1+\alpha}$-estimates and the extra property is related to the $t$-direction.

To examine the Neumann problem we need to know the analog result for the Dirichlet case (see appendix for the proof).

\begin{thm} \label{Dirichlet_H1+a} (Boundary $H^{1+\alpha}$-estimates for the Dirichlet problem).
Let $g$ be an $H^{1+\alpha}$-function locally on  $Q_{1}^*$ and $u \in C \left( Q_1^+ \cup Q_1^*  \right)$ be bounded and satisfies in the viscosity sense
\begin{align} \label{Dir_prob}
\begin{cases}
F(D^2u) -u_t = 0, &\ \ \ \ \ \text{ in } \ Q_1^+ \\
u = g, &\ \ \ \ \ \text{ on } \ Q_1^*. \\
\end{cases}
\end{align}
Then the first derivatives $u_{x_1},\dots,u_{x_{n-1}},u_y$ exist in $\overline{Q}_{1/2}^+$. Moreover there exists universal constant $0<\alpha_0 < 1$ and a polynomial $R_{1;P_0}(X)= A_{P_0}+ B_{P_0} \cdot (X-X_0)$, where $A_{P_0}=u(P_0)=g(P_0)$ and $B_{P_0}= \left( u_{x_1}(P_0), \dots, u_{x_{n-1}}(P_0), u_y(P_0) \right) = \left( g_{x_1}(P_0), \dots, g_{x_{n-1}}(P_0), u_y(P_0) \right)$ so that for $\beta=\min\{\alpha, \alpha_0\}$
\begin{equation}\label{Dirichlet_H1+a_est}
|u(X,t) - R_{1;P_0}(X)| \leq C \left( ||u||_{L^\infty \left( Q_{1}^+ \right)}+||g||_{H^{1+\alpha}\left(  \overline{Q}_{1/2}^* \right)}  +|F(O)| \right) \ p(P,P_0)^{1+\beta}
\end{equation}
for every $P=(X,t) \in \overline{Q}_{1/2}^+(P_0)$, where $C>0$ is a universal constant.
\end{thm}

In order to get (punctual) $H^{1+\alpha}$-regularity for the Neumann problem it is enough (due to Theorem \ref{Dirichlet_H1+a}) to show that the restriction of $u$ on $Q_1^*$ is locally $H^{1+\alpha}$. To do so, we need the following lemma.

\begin{lemma} \label{caffcab}
Let $0<\alpha<1$, $0<\beta \leq 1$, $0<A<B$ and $K >0$ be constants. Let $u
\in L^\infty([A,B])$ with $\norm{u}_{L^\infty([A,B])} \leq K$. Let $d=B-A$. Define, for $h \in \R$ with $0 < \abs{h} \leq \frac{d}{2}$,
\[ v_{\beta,h} (l) = \frac{u(l+h) - u(l)}{\abs{h}^\beta}, \:\:\:\:\: l \in I_h, \]
where $I_h = [A,B-h]$ if $h>0$ and $I_h = [A-h,B]$ if $h<0$.
Assume that $v_{\beta,h} \in C^\alpha(I_h)$ and
$\norm{v_{\beta,h}}_{C^\alpha(I_h)} \leq K$, for any
$0<\abs{h}\leq \frac{d}{2}$. Then we have
\begin{enumerate}
\item If $\alpha + \beta < 1$ then $u \in C^{\alpha+\beta}([A,B])$ and $\norm{u}_{C^{\alpha+\beta}([A,B])} \leq CK$.
\item If $\alpha + \beta > 1$ then $u \in C^{0,1}([A,B])$ and $\norm{u}_{C^{0,1}([A,B])} \leq CKd^{\alpha+\beta-1}$
\end{enumerate}
where the constant C depends only on $\alpha$ and $\beta$.
\end{lemma}

The above lemma is proved in \cite{CC} (Lemma 5.6) in the interval $[-1,1]$. With a rescale argument (considering $\tilde{u}(k):=u\left( \frac{d}{2} k+\frac{A+B}{2} \right)$) we can obtain Lemma \ref{caffcab}.

\begin{preremark} \label{negative_h}
Observe that if $v_{\beta,h}$ is $C^\alpha$ only for negative values of $h$ then we will have the estimates of \textit{1.} and \textit{2.} in $\left[ A+\frac{d}{2},B\right]$ and not in the whole $[A,B]$. This is useful when we study the $t$-direction. It can be deduced easily from the proof of Lemma 5.6 in \cite{CC} and a rescaling argument.
\end{preremark}

\begin{thm} \label{Neumann_H1+a} (Boundary $H^{1+\alpha}$-estimates for the Neumann problem).
Let $u \in C \left( Q_1^+ \cup Q_1^*  \right)$ be bounded and satisfies in the viscosity sense
\begin{align} \label{Neu_prob}
\begin{cases}
F(D^2u) -u_t = 0, &\ \ \ \ \ \text{ in } \ Q_1^+ \\
u_y = 0, &\ \ \ \ \ \text{ on } \ Q_1^*. \\
\end{cases}
\end{align}
Then the first derivatives $u_{x_1},\dots,u_{x_{n-1}},u_y$ exist in $\overline{Q}_{1/2}^+$. Moreover there exists a universal constant $0<\alpha < 1$ and a polynomial $R_{1;P_0}(X)= A_{P_0}+ B_{P_0} \cdot (X-X_0)$, where $A_{P_0}=u(P_0)$ and $B_{P_0}= \left( u_{x_1}(P_0), \dots, u_{x_{n-1}}(P_0), 0 \right)$ so that
\begin{equation}\label{Neumann_H1+a_est}
|u(X,t) - R_{1;P_0}(X)| \leq C \left( ||u||_{L^\infty \left( Q_{1}^+ \right)}+|F(O)|\right) \ p(P,P_0)^{1+\alpha}
\end{equation}
for every $P=(X,t) \in \overline{Q}_{1/2}^+(P_0)$, where $C>0$ is a universal constant.

In addition, $u_t$ exists and it is $H^\alpha $ in $ \overline{Q}_{1/2}^+$ with the corresponding estimate being bounded by above by a term of the form $C \left( ||u||_{L^\infty \left( Q_{1}^+ \right)} +|F(O)|\right)$ .
\end{thm}

\begin{proof} For convenience we denote by $K:=||u||_{L^\infty \left( Q_{1}^+ \right)} +|F(O)|$.

Lets examine first the $x_i$-direction, for $i=1,\dots,n-1$. For $e_i=(0,\dots, x_i=1,\dots,0) \in \R^n$, $0<\beta \leq 1$, $0<|h|<\frac{1}{8}$ we define
$$v_{\beta,h,i}(X,t)=\frac{u(X+he_i,t)-u(X,t)}{|h|^\beta}, \ \ \text{ for } \ \ (X,t) \in Q_{7/8}^+$$
(note that if $(X,t) \in Q_{7/8}^+$ then $(X+he_i,t) \in Q_1^+$). We define the following $H^\alpha$-norm which deals only with $x_i$-direction
$$||u||_{H^\alpha_i(\Omega)} :=  ||u||_{L^\infty(\Omega)} + \sup_{(X,t),(Z,t) \in \Omega \atop x_j=z_j, x_i \neq z_i} \frac{|u(X,t)-u(Z,t)|}{|x_i-z_i|^\alpha}.$$

It is easy to verify that
\begin{align*} 
\begin{cases}
v_{\beta,h,i} \in S_p \left( \frac{\lambda}{n}, \Lambda \right),& \ \ \ \text{ in } \ \ Q_{7/8}^+ \\
(v_{\beta,h,i})_y=0,& \ \ \ \text{ on } \ \ Q_{7/8}^*.
\end{cases}
\end{align*}

Now, take $0<r<\rho\leq \frac{7}{8}$. By $H^\alpha$-estimates we have
\begin{equation} \label{Neumann_H1+a_est1}
||v_{\beta,h,i}||_{H^{\alpha_1}\left( \overline{Q}_r^+ \right)} \leq C \ C(r,\rho) \ ||v_{\beta,h,i}||_{L^{\infty}\left( \overline{Q}_{\frac{r+\rho}{2}}^+ \right)}.
\end{equation}
Next, observe that if $(X,t) \in \overline{Q}_{\frac{r+\rho}{2}}^+$, once we choose $0<|h|< \frac{\rho-r}{2}$ 
we get $(X+he_i,t) \in \overline{Q}_\rho^+$. Therefore
$\ |v_{\beta,h,i}(X,t)| \leq ||u||_{H^{\beta}_i\left( \overline{Q}_\rho^+ \right)}.$
Returning to (\ref{Neumann_H1+a_est1}) we have that
\begin{equation} \label{Neumann_H1+a_est2}
||v_{\beta,h,i}||_{H^{\alpha_1}\left( \overline{Q}_r^+ \right)} \leq C \ C(r,\rho) \ ||u||_{H^{\beta}_i\left( \overline{Q}_\rho^+ \right)}
\end{equation}
for any $0<r<\rho\leq \frac{7}{8}$ and $h$ as above. Moreover observe that $H^\alpha$-estimates ensure that there exists some universal $0<\alpha_2<1$ so that for any $0<\rho<1$,
\begin{equation} \label{Neumann_H1+a_est3}
||u||_{H^{\alpha_2}_i\left( \overline{Q}_\rho^+ \right)} \leq ||u||_{H^{\alpha_2}\left( \overline{Q}_\rho^+ \right)} \leq C \ C(\rho) K.
\end{equation}

Note that we can choose some suitable $0<\alpha < \min \{ \alpha_1, \alpha_2 \}$ in order to succeed finding a universal integer $m_0 \geq 1$ so that $m_0 \alpha <1 \ \ \text{ and } \ \ (m_0+1)\alpha>1.$ Next we apply, using Lemma \ref{caffcab}, an iterative procedure which can be started from $\beta=\alpha$ and intent to finish at $\beta=1$. We consider the following finite sequence of (universal) radii
$$r_k= \frac{7}{8}- \frac{k}{16m_0}, \ \ \text{ for } \ \ k=0,1,\dots,2m_0.$$
Note that $r_0=\frac{7}{8}, r_{2m_0}=\frac{3}{4}$ and $r_{k-1}-r_k= \frac{1}{16m_0}$.
\\$\ $
\\ \hspace{-5mm}\underline{Step 1.} (of the iteration): Applying (\ref{Neumann_H1+a_est2}) together with (\ref{Neumann_H1+a_est3}) with $\beta=\alpha, r=r_1, \rho=\frac{7}{8}$ we obtain that $\ ||v_{\alpha,h,i}||_{H^{\alpha}\left( \overline{Q}_{r_1}^+ \right)} \leq CK$, for any $0<|h|<\frac{1}{16m_0}.$
Then using the above and Lemma \ref{caffcab} we shall get that $\ ||u||_{H^{2\alpha}_i\left( \overline{Q}_{r_2}^+ \right)}  \leq C K$. That is, we want, for any two $(X,t), (X+Le_i,t) \in \overline{Q}_{r_2}^+$, to have that $|u(X+Le_i,t)-u(X,t)|\leq CK|L|^{2\alpha}$. We split into two cases: If $|L| \geq \frac{1}{16m_0}$, then $|u(X+Le_i,t)-u(X,t)|\leq 2K \leq 2K (16m_0)^{2\alpha}|L|^{2\alpha} \leq CK |L|^{2\alpha}$. If $|L| < \frac{1}{16m_0}$, consider the interval $I=\left[ -\frac{1}{16m_0}, \frac{1}{16m_0} \right]$ and we define
$$\tilde{u}^{(X,t),i}(l)=u(X+le_i,t), \ \ \text{ for } \ \ l \in I.$$
In addition let $\tilde{v}_{\alpha,h}^{(X,t),i}(l),\ l \in I_h$, for $0<|h|<\frac{1}{16m_0}$ be as in Lemma \ref{caffcab}. Observe that $\tilde{v}_{\alpha,h}^{(X,t),i}(l) = v_{\alpha,h,i}(X+le_i,t).$ Now, if $(X,t) \in \overline{Q}_{r_2}^+$ and $l \in I$ then $(X+le_i,t) \in \overline{Q}_{r_1}^+$. Hence $||\tilde{v}_{\alpha,h}^{(X,t),i}||_{C^{\alpha}\left( I_h \right)} \leq ||v_{\alpha,h,i}||_{H^{\alpha}\left( \overline{Q}_{r_1}^+ \right)} \leq C K.$ Therefore, Lemma \ref{caffcab} implies $||\tilde{u}^{(X,t),i}||_{C^{\alpha}\left( I \right)} \leq CK$ (note that the length of $I$ is a universal number). Then, since $0, L \in I$, we have the desired.
\\$\ $

\hspace{-5mm}\underline{Step $m_0$.} (of the iteration): Applying (\ref{Neumann_H1+a_est2}) with $\beta=m_0\alpha, r=r_{2m_0-1}, \rho=r_{2m_0-2}$ together with Step $m_0-1$ we obtain that $||v_{m_0\alpha,h,i}||_{H^{\alpha}\left( \overline{Q}_{r_{2m_0-1}}^+ \right)} \leq CK$, for any $0<|h|<\frac{1}{16m_0}.$ Then again as in Step 1 (using Lemma \ref{caffcab}) and recalling the choice of constants $\alpha$ and $m_0$ ($(m_0+1)\alpha>1$) we can derive that $||u||_{H^{1}_i\left( \overline{Q}_{3/4}^+ \right)}  \leq C K$.

This last estimate ensures the existence of $u_{x_i}$ on $Q_{\frac{3}{4}}^*$ for any $i=1,\dots,n-1$. Moreover, applying again  (\ref{Neumann_H1+a_est2}) with $\beta=1, r=\frac{5}{8}, \rho=\frac{3}{4}$ together with the above we conclude that $||v_{1,h,i}||_{H^{\alpha}\left( \overline{Q}_{5/8}^+ \right)} \leq CK, \ \ \text{ for any } \ \ 0<|h|<\frac{1}{16m_0}$, which gives a suitable $H^\alpha$-estimate for $u_{x_i}$ on $Q_{5/8}^*$.

Now, observing that $u$ satisfies, in the viscosity sense, a problem of the form (\ref{Dir_prob}) with $g(x,t)=u(x,0,t)$ and since $g$ is $H^{1+\alpha}$-function on $Q_{5/8}^*$ we can apply Theorem \ref{Dirichlet_H1+a} to get the desired result for $X$-directions.

It remains to examine the $t$-direction. The proof follows the same lines as above under minor modifications. We present the proof briefly for completeness.

So for $0<\beta \leq 2$, $-\frac{1}{8}<h<0$ we define
$$v_{\beta,h}(X,t)=\frac{u(X,t+h)-u(X,t)}{|h|^\frac{\beta}{2}}, \ \ \text{ for } \ \ (X,t) \in Q_{7/8}^+.$$
We define the following $H^\alpha$-norm which deals only with $t$-direction
$$||u||_{H^\alpha_t(\Omega)} :=  ||u||_{L^\infty(\Omega)} + \sup_{(X,t),(X,s) \in \Omega, t \neq s} \frac{|u(X,t)-u(X,s)|}{|t-s|^{\frac{\alpha}{2}}}.$$
Note that we can easily obtain that
\begin{align*} 
\begin{cases}
v_{\beta,h} \in S_p \left( \frac{\lambda}{n}, \Lambda \right),& \ \ \ \text{ in } \ \ Q_{7/8}^+ \\
(v_{\beta,h})_y=0,& \ \ \ \text{ on } \ \ Q_{7/8}^*.
\end{cases}
\end{align*}
Then 
\begin{equation}\label{Neumann_H1+a_est11}
||v_{\beta,h}||_{H^{\alpha}\left( \overline{Q}_r^+ \right)} \leq C \ C(r,\rho) \ ||u||_{H^{\beta}_t\left( \overline{Q}_\rho^+ \right)}
\end{equation}
for any $0<r<\rho\leq \frac{7}{8}, -\left( \frac{\rho-r}{2} \right)^2<h< 0$. Moreover for any $0<\rho<1$
\begin{equation} \label{Neumann_H1+a_est10}
||u||_{H^{\alpha}_t\left( \overline{Q}_\rho^+ \right)} \leq ||u||_{H^{\alpha}\left( \overline{Q}_\rho^+ \right)} \leq C \ C(\rho) K.
\end{equation}

We take $\alpha$ small enough so that there exists a universal integer $m_0$ which satisfies $\frac{m_0\alpha}{2}<1$ and $(m_0+1)\frac{\alpha}{2}>1$. For the iteration consider the following finite sequence of (universal) radii
$$r_k= \frac{7}{8}- \frac{k}{16m_0}, \ \ \text{ for } \ \ k=0,1,\dots,2m_0.$$
Note that $r_0=\frac{7}{8}, r_{2m_0}=\frac{3}{4}$ and $r_{k-1}-r_k= \frac{1}{16m_0}$.
\\$\ $
\\ \hspace{-5mm}\underline{Step 1.} (of the iteration): Applying (\ref{Neumann_H1+a_est11}) together with (\ref{Neumann_H1+a_est10}) we obtain that $||v_{\alpha,h}||_{H^{\alpha}\left( \overline{Q}_{r_1}^+ \right)} \leq CK$, for any $-\left(\frac{1}{16m_0}\right)^2<h<0$. Using the above and Remark \ref{negative_h} we shall get
$||u||_{H^{2\alpha}_t\left( \overline{Q}_{r_2}^+ \right)}  \leq C K$.
That is, we take any two $(X,t_1) \neq (X,t_2) \in \overline{Q}_{r_2}^+$ and since $t_1 \neq t_2$ we can assume without the loss of generality that $t_1>t_2$ and denote by $t:=t_1$ and $t+L:=t_2$ (then $L=t_2-t_1<0$) and we aim to get that $|u(X,t)-u(X,t+L)|\leq CK|L|^{\alpha}$. We split into two cases: If $|L| \geq \frac{1}{2} \left(\frac{1}{16m_0} \right)^2$, then $|u(X,t)-u(X,t+l)|\leq 2K \leq 2K 2^\alpha(16m_0)^{2\alpha}|L|^{\alpha} \leq CK |L|^{\alpha}$. If $|L| < \frac{1}{2} \left(\frac{1}{16m_0} \right)^2$, we consider the interval $I=\left[ -\left(\frac{1}{16m_0} \right)^2, 0\right]$. Define
$$\tilde{u}^{(X,t)}(l)=u(X,t+l), \ \ \text{ for } \ \ l \in I$$
and $\tilde{v}_{\frac{\alpha}{2},h}^{(X,t)}(l)= \frac{\tilde{u}^{(X,t)}(l+h)-\tilde{u}^{(X,t)}(l)}{|h|^{\frac{\alpha}{2}}}$, for $-\frac{1}{2} \left(\frac{1}{16m_0} \right)^2<h<0, l \in I_h$ where $I_h$ is as in Lemma \ref{caffcab}. Then $\tilde{v}_{\frac{\alpha}{2},h}^{(X,t)}(l)= v_{\alpha,h}(X,t+l)$. Now, if $(X,t) \in \overline{Q}_{r_2}^+$, $l \in I$ then $-\left(\frac{1}{16m_0} \right)^2-r_2^2<t+l \leq l <0$. But, $-\left(\frac{1}{16m_0} \right)^2-r_2^2=-r_1^2+2r_1r_2 -2r_2^2\geq -r_1^2$ (using that $r_1>r_2$), i.e. $(X,t+l) \in \overline{Q}_{r_1}^+$. Then, for $l_1, l_2 \in I_h$, $\abs{\tilde{v}_{\frac{\alpha}{2},h}^{(X,t)}(l_1)-\tilde{v}_{\frac{\alpha}{2},h}^{(X,t)}(l_2)} \leq CK |l_1-l_2|^{\frac{\alpha}{2}}$. Then Remark \ref{negative_h} implies $||\tilde{u}^{(X,t)}||_{C^{\alpha}\left( \tilde{I} \right)} \leq CK$, where $\tilde{I}= \left[ -\frac{1}{2} \left( \frac{1}{16m_0}\right)^2,0 \right]$. Since $0, L \in \tilde{I}$, we have the desired.
\\$\ $

\hspace{-5mm}\underline{Step $m_0$.} (of the iteration): Applying (\ref{Neumann_H1+a_est10}) together with Step $m_0-1$ we obtain that $||v_{m_0\alpha,h}||_{H^{\alpha}\left( \overline{Q}_{r_{2m_0-1}}^+ \right)} \leq CK$, for any $-\left(\frac{1}{16m_0} \right)^2<h<0.$ Then as in Step 1 (using Remark \ref{negative_h}) and recalling how the constants $\alpha$ and $m_0$ have been chosen ($(m_0+1)\frac{\alpha}{2}>1$) we can derive that $\ ||u||_{H^{2}_t\left( \overline{Q}_{3/4}^+ \right)}  \leq C K$.

This last estimate ensures the  existence of $u_t$ in $Q_{\frac{3}{4}}^+$. Moreover, by applying again (\ref{Neumann_H1+a_est10}) together with the above gives
$$||v_{1,h}||_{H^{\alpha}\left( \overline{Q}_{5/8}^+ \right)} \leq CK, \ \ \text{ for any } \ \ -\left(\frac{1}{16m_0} \right)^2<h<0.$$
\end{proof}

\subsection{$H^{1+\alpha}$-estimates for the oblique derivative case} \label{oblique_H^1+a_section}

First we examine a constant oblique derivative problem using the change of variables of section \ref{changeofvar}. In the following we assume for convenience that $F(O)=0$ but note that this assumption is not essential in the sense that we can find an operator with the same ellipticity constants satisfying this assumption and up to a subtraction of a paraboloid, $u$ will satisfy the new equation.

\begin{thm} \label{oblique_constant_H1+a} (Boundary $H^{1+\alpha}$-estimates for the constant oblique derivative problem).
Let $u \in C \left( Q_1^+ \cup Q_1^*  \right)$ be bounded and satisfies in the viscosity sense
\begin{align} \label{oblique_constant_prob}
\begin{cases}
F(D^2u) -u_t = 0, &\ \ \ \ \ \text{ in } \ Q_1^+ \\
\beta \cdot Du = 0, &\ \ \ \ \ \text{ on } \ Q_1^* \\
\end{cases}
\end{align}
where $\beta$ is a constant function. Then the first derivatives $u_{z_1},\dots,u_{z_{n-1}},u_w$ exist at $(0,0)$. Moreover there exists a universal constant $0<\alpha < 1$ and a polynomial $R_{1}(Z)= A^0+ B^0 \cdot Z$, where $A^0=u(0,0)$ and $B^0= Du(0,0) \in \R^n$ (then, $\beta \cdot B^0=0$) so that
\begin{equation}\label{oblique_constant_H1+a_est}
|u(Z,t) - R_{1}(Z)| \leq C \ ||u||_{L^\infty \left( Q_{1}^+ \right)} \ \left(|Z|+|t|^{1/2}\right)^{1+\alpha}
\end{equation}
for every $P=(Z,t) \in \overline{Q}_{\rho}^+$, where $C>0$, $0<\rho<1$ are universal constants.

In addition, $u_t$ exists and it is $H^\alpha $ in $ \overline{Q}_{\rho}^+$ with the corresponding estimate being bounded by above by a term of the form $C \ ||u||_{L^\infty \left( Q_{1}^+ \right)}$.
\end{thm}

\begin{proof}
Let $A$ be the transformation defined in section \ref{changeofvar}. Define $v(X,t)=u(AX,t)$, for $(X,t) \in Q_r^+$, where $0<r<\frac{\delta_0}{\delta_0+1}<1$. Note that $Q_r^+ \subset \tilde{Q_1^+}$. Then 
\begin{align*} 
\begin{cases}
\tilde{F}(D^2v) -v_t = 0, &\ \ \ \ \ \text{ in } \ \tilde{Q_r^+} \\
v_y = 0, &\ \ \ \ \ \text{ on } \ Q_r^*.\\
\end{cases}
\end{align*}
So applying Theorem \ref{Neumann_H1+a} to $v$ we have that $v_{x_1},\dots,v_{x_{n-1}},v_y$ exist at $(0,0)$ and there exists a polynomial $\tilde{R}_{1}(X)=\tilde{A}^0+\tilde{B}^0 \cdot X$, where $\tilde{A}^0=v(0,0)$ and $\tilde{B}^0= \left( v_{x_1}(0,0), \dots, v_{x_{n-1}}(0,0), 0 \right)$  so that 
$$|v(X,t) - \tilde{R}_1(X)| \leq C \ ||v||_{L^\infty \left( Q_{r}^+ \right)} \ \left(|X|+|t|^{1/2}\right)^{1+\alpha}$$
for every $(X,t) \in \overline{Q}_{r/2}^+$, where $C>0$, $0<\alpha<1$ are universal constants. In addition, $v_t$ exists and it is $H^\alpha $ in $ \overline{Q}_{r/2}^+$ with the corresponding estimate being bounded by above by a term of the form $C \ ||v||_{L^\infty \left( Q_{r}^+ \right)}$.

Let $R_1(Z)=\tilde{R}_1(A^{-1}Z)=\tilde{A}^0+\tilde{B}^0\cdot A^{-1}Z=\tilde{A}^0+(A^{-1})^\tau \tilde{B}^0\cdot Z$ and observe that $\tilde{A}^0=v(0,0)=u(0,0)=:A^0$
and
\begin{align*}
(A^{-1})^\tau \tilde{B}^0&=\left( v_{x_1}(0,0),\dots,v_{x_{n-1}}(0,0), v_y(0,0)- \frac{\beta _1}{\beta_n}v_{x_1}(0,0)- \dots - \frac{\beta _{n-1}}{\beta_n}v_{x_{n-1}}(0,0)  \right) \\ 
&=\left( u_{z_1}(0,0),\dots,u_{z_{n-1}}(0,0), u_w(0,0)  \right)=:B^0
\end{align*}
Note that for $\rho=\frac{\delta_0r}{2(\delta_0+1)}<1$ if $(Z,t) \in Q_\rho^+$ then $(A^{-1}Z,t) \in \overline{Q}_{r/2}^+$, so
$$|u(Z,t) - R_1(Z)| \leq C \ ||u||_{L^\infty \left( Q_{r}^+ \right)} \ \left(|Z|+|t|^{1/2}\right)^{1+\alpha}$$
for every $(Z,t) \in \overline{Q}_{\rho}^+$. Furthermore  $u_t(Z,t)=v_t(A^{-1}Z,t)$ and $\norm{u_t}_{H^\alpha\left(Q_\rho^+\right)} \leq C \ \norm{u}_{L^\infty \left(Q_1^+\right)}$.
\end{proof}

\begin{thm} \label{oblique_general_H1+a} (Boundary $H^{1+\alpha}$-estimates for the general oblique derivative problem).
Let $g$ and $\beta$ be $H^{\gamma}$ locally on  $Q_{1}^*$, $f \in L^q \left( Q_1^+\right)$ with $q > \frac{(n+1)(n+2)}{2}$ and $u \in C \left( Q_1^+ \cup Q_1^*  \right)$ be bounded and satisfy in the viscosity sense
\begin{align*}
\begin{cases}
F(D^2u) -u_t = f, &\ \ \ \ \ \text{ in } \ Q_1^+ \\
\beta \cdot Du = g, &\ \ \ \ \ \text{ on } \ Q_1^*. \\
\end{cases}
\end{align*}
Then the first derivatives $u_{x_1},\dots,u_{x_{n-1}},u_y$ exist at $(0,0)$. Moreover there exists universal constant $0<\alpha_0 < 1$ and a polynomial $R_{1;0}(X)= A^{0}+ B^{0} \cdot X$, where $A^0=u(0,0)$ and $B^0= Du(0,0) \in \R^n$ so that
\begin{equation} \label{oblique_general_H1+a_est} 
|u(X,t) - R_{1;0}(X)| \leq C \left( ||u||_{L^\infty \left( Q_{1}^+ \right)}+||g||_{H^{\gamma}\left(  \overline{Q}_{1/2}^* \right)} +||f||_{L^q\left(  \overline{Q}_{1}^+ \right)} \right) \ \left(|X|+|t|^{1/2}\right)^{1+\alpha}
\end{equation}
for every $(X,t) \in \overline{Q}_{1/4}^+$, where $C>0$ is a universal constant.
\end{thm}

Note that we may assume that $u(0,0)=0$, considering $u(X,t)-u(0,0)$ and that $g(0,0)=0$, considering $u(X,t) - \frac{g(0,0) \ y}{\beta_n(0,0)}$.

\begin{proof}
For convenience let us denote $K:=||u||_{L^\infty \left( Q_{1}^+ \right)}+||g||_{H^{\gamma}\left(  \overline{Q}_{1/2}^* \right)} +||f||_{L^q\left(  \overline{Q}_{1}^+ \right)}$ and $\beta^0:=\beta (0,0) \in \R^n$.

We intend to find some $B^0 \in \R^n$, with  $\beta^0 \cdot B^0 =0$ so that for universal $C>0, 0<\eta<1, 0<\rho<1, \alpha_0>0$ and $\alpha = \min\{\alpha_0, \gamma, \frac{2q-(n+1)(n+2)}{q(n+1)}\}$ we will have
\begin{equation} \label{oblique_general_H1+a_est1} 
\osc_{Q_{\rho \eta^k}^+} \left( u(X,t) - B^0 \cdot X \right) \leq CK \eta^{k(1+\alpha)}, \ \ \text{ for any } \ \ k \in \N. 
\end{equation} 

Now, to prove (\ref{oblique_general_H1+a_est1} ) we are going to show by induction that there exist universal constants $0<\eta <1, 0<\rho<1, \bar{C}>0, \alpha_0>0$ such that for $\alpha = \min\{\alpha_0, \gamma, \frac{2q-(n+1)(n+2)}{q(n+1)}\}$ we can find a vector $B_k \in \R^n$, with $\beta^0 \cdot B_k=0$ for any $k \in \N$ so that
\begin{equation} \label{oblique_general_H1+a_est2} 
\osc_{Q_{\rho \eta^k}^+} \left( u(X,t) - B_k \cdot X \right) \leq \bar{C}K \eta^{k(1+\alpha)}
\end{equation}
and
\begin{equation} \label{oblique_general_H1+a_est3} 
|B_{k+1}-B_k| \leq CK\eta^{k\alpha}.
\end{equation}
Note that the correct constants will be deduced from the induction. The details follow.

First, for $k=0$, take $B_0=0$ and choose any $\bar{C}\geq 2$. Next for the induction we assume that we have found vectors $B_0,B_1,\dots,B_{k_0}$ for which (\ref{oblique_general_H1+a_est2}) and (\ref{oblique_general_H1+a_est3}) are true. Denoting by $r:=\frac{\rho \eta ^{k_0}}{2}$ and $B:=B_{k_0}$ we have $\beta^0 \cdot B=0$ and
\begin{equation}\label{oblique_general_H1+a_est4} 
\osc_{Q_r^+} \left( u(X,t) - B \cdot X \right) \leq \frac{4}{\rho^{1+\alpha}} \bar{C}K r^{(1+\alpha)}.
\end{equation}

Now we are going to consider a suitable constant oblique derivative problem (as the one of Theorem \ref{oblique_constant_H1+a}). So let $v$ be the viscosity solution of
\begin{align*}
\begin{cases}
F(D^2v) -v_t = 0, &\ \ \ \ \ \text{ in } \ Q_r^+ \\
\beta^0 \cdot Dv=0, &\ \ \ \ \ \text{ on } \ Q_r^* \\
v=u-B \cdot X, &\ \ \ \ \ \text{ on } \ \partial_p Q_r^+ \setminus Q_r^*.
\end{cases}
\end{align*}
Then $v$ satisfies  ABPT-estimate for the oblique derivative case (see Theorem \ref{ABP-Oblique}) which gives
\begin{equation} \label{oblique_general_H1+a_est6} 
\osc_{Q_{r}^+} v \leq \osc_{Q_{r}^+}\left( u(X,t) - B \cdot X \right).
\end{equation}
From Theorem \ref{oblique_constant_H1+a} we also have that $\bar{B}=Dv(0,0)$ exists and $\beta^0 \cdot \bar{B} =0$. Moreover
\begin{equation} \label{oblique_general_H1+a_est7} 
\osc_{Q_{\tilde{r}}^+} \left(v(X,t)-\bar{B} \cdot X \right) \leq C_0 \left( \frac{\tilde{r}}{r}\right)^{1+\alpha_1} \ \osc_{Q_{r}^+} v 
\end{equation}
for any $\tilde{r} \leq \rho \ r$, where $0<\rho<1$ universal and $|\bar{B}| \leq  \frac{C}{r}\osc_{Q_{r}^+} v$. Next, we take $\tilde{r}=\eta r$ (for $0<\eta<\rho$) in (\ref{oblique_general_H1+a_est7} ). Hence
\begin{equation} \label{oblique_general_H1+a_est9} 
\osc_{Q_{\eta r}^+} \left(v(X,t)- \bar{B} \cdot X \right) \leq C_0 \eta^{1+\alpha_1} \osc_{Q_{r}^+} v.
\end{equation}
Now take (universal) $0<\eta<<1$ sufficiently small in order to have that $8 \ C_0 \eta^{\alpha_1}<1$. We denote by $1-\theta :=8\ C_0 \eta^{\alpha_1}$, where $0<\theta<1$ is a universal constant. Then 
\begin{equation} \label{oblique_general_H1+a_est10} 
\osc_{Q_{\eta r}^+} \left(v(X,t)-\bar{B} \cdot X \right) \leq \frac{(1-\theta)}{2} \ \eta \ \frac{\bar{C}}{\rho^{1+\alpha}}Kr^{1+\alpha}.
\end{equation}

Now to return to $u$ we define $w=u-B \cdot X-v$. Then
\begin{align*}
\begin{cases}
w \in S_p\left( \frac{\lambda}{n}, \Lambda ,f \right), &\ \ \ \ \ \text{ in } \ Q_r^+ \\
\beta \cdot Dw =g- \beta \cdot \left( B + Dv \right), &\ \ \ \ \ \text{ on } \ Q_r^* \\
w=0, &\ \ \ \ \ \text{ on } \ \partial_p Q_r^+ \setminus Q_r^*.
\end{cases}
\end{align*}

Now for $0<\mu <1$ (to be chosen universal) we denote by $\bar{r}:= r(1-\mu )<r$. We apply again Theorem \ref{ABP-Oblique}
\begin{align}\label{oblique_general_H1+a_est11}
\osc_{Q_{\bar{r}}^+} w &\leq Cr ||f||_{L^{n+1}(Q^+_r)}+ Cr ||g||_{L^\infty(Q^*_r)}+Cr ||\beta \cdot B||_{L^\infty(Q^*_r)} \nonumber \\
&\ \ \ +Cr ||\beta \cdot Dv||_{L^\infty(Q^*_{\bar{r}})}+ \osc_{{\partial_p Q_{\bar{r}}^+ \setminus Q_{\bar{r}}^*}} w \nonumber \\
&=: \textbf{I}+\textbf{II}+\textbf{III}+\textbf{IV}+\textbf{V}.
\end{align}
We want to bound all five terms by something of order $r^{1+\alpha}$. We start with term \textbf{I}. Using H\"older inequality and that $q >\frac{(n+1)(n+2)}{2}>n+1$ we get $\textbf{I} \leq C \ r^{1+\left(1 - \frac{n+2}{q} \right)} \ K$.
Next, for term \textbf{II}, we use the $H^\gamma$-regularity of $g$ and the fact that $g(0,0)=0$, then $\textbf{II}  = Cr||g-g(0,0)||_{L^\infty(Q^*_r)} \leq Cr^{1+\gamma}K.$ We continue with term \textbf{III}. We use the $H^\gamma$-regularity of $\beta$ and the fact that $\beta^0 \cdot B=0$, $\textbf{III} \leq Cr||\beta-\beta^0||_{L^\infty(Q^*_r)}\ |B| \leq Cr^{1+\gamma}K$,
where we have used that $|B|\leq CK$ which can be derived from (\ref{oblique_general_H1+a_est3}) and the fact that $|B_0|=0$. Next for term \textbf{IV}, we use again the $H^\gamma$-regularity of $\beta$ and the fact that $\beta^0 \cdot Dv=0$ on $Q_r^*$, we have $\textbf{IV} \leq Cr||\beta-\beta^0||_{L^\infty(Q^*_r)}\ ||Dv||_{L^\infty(Q^*_{\bar{r}})} \leq C_2 \rho^{\gamma} \frac{\bar{C}}{\rho^{1+\alpha}}Kr^{1+\alpha}.$ Finally we examine term \textbf{V}. Let $(X_0,t_0) \in \partial_p Q_{\bar{r}}^+ \setminus Q_{\bar{r}}^*$. If $|X_0|=\bar{r}$ we choose $\bar{X}_0 \in \left(\partial B_r\right)^+$ so that $|X_0-\bar{X}_0|=\mu r \leq \sqrt{2 \mu}r$ and $\bar{t}_0=t_0$. If $|X_0|<\bar{r}$ then $t_0=-(1-\mu)^2r^2$ and we choose $\bar{t}_0=-r^2$ then $|t_0-\bar{t}_0|^{1/2}=r \sqrt{\mu (2- \mu)} \leq \sqrt{2 \mu}r$ and $\bar{X}_0=X_0$. In any case  $|X_0-\bar{X}_0|+|t_0-\bar{t}_0|^{1/2}\leq  \sqrt{2 \mu}r $ and $(\bar{X}_0,\bar{t}_0) \in \partial_p Q_{r}^+ \setminus Q_{r}^*$ that is $w \left( \bar{X}_0, \bar{t}_0 \right) =0$. Then
\begin{align}\label{oblique_general_H1+a_est11_1}
|w \left( X_0,t_0 \right)|\leq |\left( u \left( X_0,t_0\right) - B \cdot X_0 \right)-\left( u \left( \bar{X}_0, \bar{t}_0 \right) -B \cdot \bar{X}_0 \right)|+|v \left( X_0,t_0\right)-v\left( \bar{X}_0, \bar{t}_0 \right)|
\end{align}
and we bound these terms using $H^\alpha$-estimates. Indeed, we have that
\begin{align*}
\begin{cases}
F(D^2(u-B\cdot X)) -(u-B\cdot X)_t = f, &\ \ \ \ \ \text{ in } \ Q_{2r}^+ \\
\beta \cdot D(u-B\cdot X) = g-\beta \cdot B, &\ \ \ \ \ \text{ on } \ Q_{2r}^*. \\
\end{cases}
\end{align*}
Then Theorem \ref{bdary_C^a_oblique} gives
\begin{align*}
||u-B\cdot X||_{H^{\alpha_2} \left( \overline{Q}^+_{r} \right) } \leq &\ \frac{C}{r^{\alpha_2}} \ \ ||u-B\cdot X||_{L^\infty \left( Q_{2r}^+ \right)} \\
&+\frac{C}{r^{\alpha_2}} \left( r^{\frac{n}{n+1}}||f||_{L^{n+1}\left( Q_{2r}^+ \right)}+r||g||_{L^\infty \left( Q_{2r}^* \right)} +r||\beta \cdot B||_{L^\infty \left( Q_{2r}^* \right)}\right).
\end{align*}
Next we apply global $H^\alpha$-estimates (see \cite{Wang2}) for $v$. Note that the values of $v$ on the parabolic boundary equal to $u-B\cdot X$ which is $H^{\alpha_2}$. So, for $0<\alpha_3 < <\alpha_2$ universal, 
\begin{align*}
||v||_{H^{\alpha_3} \left( \overline{Q}^+_{r} \right) } &\leq \ \frac{C}{r^{\alpha_3}} \ \left( ||v||_{L^\infty \left( Q_{r}^+ \right)}+r^{\alpha_2}||u-B\cdot X||_{H^{\alpha_2} \left( \overline{Q}^+_{r} \right) } \right)\\
&\leq \ \frac{C}{r^{\alpha_3}} \ \ ||u-B\cdot X||_{L^\infty \left( Q_{2r}^+ \right)} \\
&\ \ \ +\frac{C}{r^{\alpha_3}} \left( r^{\frac{n}{n+1}}||f||_{L^{n+1}\left( Q_{2r}^+ \right)}+r||g||_{L^\infty \left( Q_{2r}^* \right)} +r||\beta \cdot B||_{L^\infty \left( Q_{2r}^* \right)}\right).
\end{align*}
Now (\ref{oblique_general_H1+a_est11_1}) yields
\begin{align*}
|w(X_0,t_0)| &\leq C \mu ^{\alpha_3/2}\ \ ||u-B\cdot X||_{L^\infty \left( Q_{2r}^+ \right)} \\
&\ \ \ + C \mu ^{\alpha_3/2} \left( r^{\frac{n}{n+1}}||f||_{L^{n+1}\left( Q_{2r}^+ \right)}+r||g||_{L^\infty \left( Q_{2r}^* \right)} +r||\beta \cdot B||_{L^\infty \left( Q_{2r}^* \right)}\right) \\
&\leq \textbf{VI}+\textbf{I}'+\textbf{II}'+\textbf{III}'.
\end{align*}
For term \textbf{VI}, we use the hypothesis of the induction,  (\ref{oblique_general_H1+a_est4}), $\textbf{VI} \leq C_1 \mu ^{\alpha_3/2} \frac{\bar{C}}{\rho^{1+\alpha}} Kr^{1+\alpha}$. Moreover $\textbf{I}'  \leq C \mu ^{\alpha_3/2}  r^{1+\alpha(n,q)} \ K$, where $\alpha(n,q) := \frac{2q-(n+1)(n+2)}{q(n+1)}>0$. Note also that $\alpha(n,q) < 1 - \frac{n+2}{q}$.
Also, terms \textbf{II}$'$ and \textbf{III}$'$ are in fact the same as terms \textbf{II} and \textbf{III}. That is,
$$\textbf{V} \leq C_1 \mu ^{\alpha_3/2} \frac{\bar{C}}{\rho^{1+\alpha}}Kr^{1+\alpha}+C \mu ^{\alpha_3/2}  r^{1+\alpha(n,q) } \ K
+ C \mu ^{\alpha_3/2}  r^{1+\gamma } \ K.$$
So, returning to (\ref{oblique_general_H1+a_est11}), we have
$$\osc_{Q_{\bar{r}}^+} w \leq C K  r^{1+\alpha(n,q) } +C K r^{1+\gamma }+C_1 \mu ^{\alpha_3/2} \frac{\bar{C}}{\rho^{1+\alpha}}Kr^{1+\alpha} + C_2 \rho^{\gamma} \frac{\bar{C}}{\rho^{1+\alpha}}Kr^{1+\alpha}. $$

Next combine the above with (\ref{oblique_general_H1+a_est10}) and choose $\mu < 1 - 2\eta$ (then $2\eta<1-\mu$) 
\begin{align}\label{oblique_general_H1+a_est12} 
&\osc_{Q_{ \eta r}^+} \left[ u(X,t)-(B+\bar{B}) \cdot X \right] \leq \nonumber \\ &\frac{1}{2}(1-\theta) \eta \frac{\bar{C}}{\rho^{1+\alpha}}Kr^{1+\alpha}  +C K  r^{1+\alpha(n,q) } +C K r^{1+\gamma }+C_1 \mu ^{\alpha_3/2} \frac{\bar{C}}{\rho^{1+\alpha}}Kr^{1+\alpha}+C_2 \rho^{\gamma} \frac{\bar{C}}{\rho^{1+\alpha}}Kr^{1+\alpha}.
\end{align}

We choose the right constants $\alpha_0$, $\mu$ and $\bar{C}$. So, take $\alpha_0$ so that $\eta^{\alpha_0}=1-\frac{\theta}{2}$ and $\alpha=\min\{\alpha_0, \gamma, \frac{2q-(n+1)(n+2)}{q(n+1)}\}$. Take $\mu \leq \frac{\eta^{\frac{2(1+\alpha)}{\alpha_3}}}{(4C_1)^{\frac{2}{\alpha_3}}}$, $\rho \leq \frac{\eta^{\frac{1+\alpha}{\gamma}}}{(4C_2)^{\frac{1}{\gamma}}}$ and $\bar{C}$ large enough so that $\frac{\eta \theta \bar{C}}{4\rho^{1+\alpha}} \geq 2 C$ (note that our choices are all independent of $k_0$). Then we return to (\ref{oblique_general_H1+a_est12}) writing $1-\theta$ as $1-\frac{\theta}{2}-\frac{\theta}{2}$ and recalling that $r=\frac{\rho \eta^{k_0}}{2}\leq\eta^{k_0} $,
\begin{align*}
\osc_{Q_{ \rho \eta^{k_0+1}}^+} \left[ u(X,t)-(B+\bar{B}) \cdot X \right] &\leq K \left[\frac{1}{2} \left(1-\frac{\theta}{2}\right)\bar{C}\eta \eta^{k_0(1+\alpha)} +2C r^{1+\alpha } +\bar{C}\frac{\eta^{1+\alpha}}{2} \eta^{k_0(1+\alpha)}\right]\\
& \ \ \ -K\frac{\eta \theta \bar{C}}{4\rho^{1+\alpha}} r^{1+\alpha} \leq \bar{C}K \eta^{(k_0+1)(1+\alpha)}.
\end{align*}
We choose $B_{k_0+1}=B+\bar{B}$, then the above is (\ref{oblique_general_H1+a_est2}) for $k_0+1$. Also  $\beta^0 \cdot B_{k_0+1} = 0$ and $|B_{k_0+1}-B_{k_0}|=|\bar{B}| \leq \frac{C}{r} \bar{C}Kr^{1+\alpha}  \leq CK r^\alpha $.

Finally, it remains to get estimate (\ref{oblique_general_H1+a_est1}). Observe that (\ref{oblique_general_H1+a_est3}) ensures the existence of the limit $B_{\infty}:= \lim_{k\to \infty} B_k$ and this is the vector $B^0$ of (\ref{oblique_general_H1+a_est1}). Indeed, $\beta^0 \cdot B^\infty=0$ and for any $k \in \N$ we have
\begin{align*}
\osc_{Q_{\rho \eta^k}^+} \left( u(X,t)-B_\infty \cdot X \right) &\leq \osc_{Q_{\rho \eta^k}^+} \left( u(X,t)-B_k \cdot X \right)+ \eta^k (B_k-B_\infty)\\
&\leq \bar{C}K \eta^{k(1+\alpha)} + CK \eta^k \sum_{j=k}^\infty \eta^{j\alpha} \leq \bar{C}K \eta^{k(1+\alpha)} + CK \eta^k \frac{\eta^{k\alpha}}{1-\eta^\alpha} \\
&\leq CK \eta^{k(1+\alpha)}
\end{align*}
and the proof is complete.
\end{proof}

\section{H\"older Estimates for the second derivatives}\label{C2a}
\subsection{$H^{2+\alpha}$-estimates for the homogeneous Neumann case}
Here we prove $H^{2+\alpha}$-estimates. For, we will use first Lemma \ref{osc_u/y} which applied on the derivative $u_y$ will give the existence and H\"older continuity of $u_{yy}$. Then for the tangential directions, our purpose is to consider the restriction of $u$ on the thin-cylinder $Q_1^*$ and show that satisfies a suitable parabolic equation there. Hence we will be able to use the interior estimates proved in \cite{Wang2}.

First let us formulate here Theorem 1.1 of \cite{Wang2} in the form we are going to use. For operators that depend on $(X,t)$ we define
$$\theta_F(X,t) = \sup_{M \in S_n} \frac{|F(M,(X,t))-F(M,(0,0))|}{|M|+1}.$$ 

\begin{thm} \label{interior_H2+a_2} (Interior $H^{2+\alpha}$-estimates for more general operators).
Let $u \in C(Q_1)$ be a bounded viscosity solution of $F(D^2u,(X,t))-u_t=0$ in $Q_1$. Assume that any solution $v$ of the equation $F(D^2v+B,(0,0))-v_t=E$, where $B, E$ are such that $F(B,(0,0))=E$, satisfies $H^{2+\beta}$-estimates
\begin{equation} \label{interior_H2+a_2_est1}
\norm{u}_{H^{2+\beta}\left( Q_{r/2}\right)} \leq \frac{C}{r^{2+\alpha}} \left( \norm{u}_{L^\infty\left( Q_{r}\right)}+|F(O,(0,0))|\right).
\end{equation}
Assume also that 
\begin{equation} \label{interior_H2+a_2_est2}
\left( \frac{1}{m_{n+1}(Q_r)}\int_{Q_r} \theta_F ^{n+1} \right) ^{1/(n+1)} \leq C r^\alpha.
\end{equation}
Then $u_t$ and the second derivatives of $u$ exist in $\overline{Q}_{1/2}$. Moreover there exists universal constant $0<\alpha < \beta$ and a polynomial $R_{2;P_0}(X,t)= A_{P_0}+ B_{P_0} \cdot (X-X_0)+C_{P_0} (t-t_0) + \frac{1}{2} (X-X_0)^\tau D_{P_0} (X-X_0)$, where $A_{P_0}=u(P_0), B_{P_0}= \nabla_Xu(P_0), C_{P_0}=u_t(P_0)$ and $D_{P_0}:=  D^2_Xu(P_0)$, for $P_0 \in Q_{1/2}$, so that
\begin{equation}\label{interior_H2+a_2_est3}
|u(X,t) - R_{2;P_0}(X,t)| \leq C \left( ||u||_{L^\infty \left( Q_{1}^+ \right)}+|F(O,(0,0))|\right) \ p(P,P_0)^{2+\alpha}
\end{equation}
for every $P=(X,t) \in \overline{Q}_{1/2}(P_0)$, where $C>0$ is a universal constant.
\end{thm}

\begin{lemma} \label{osc_u/y}
Let $f$ be bounded in $Q_1^+$ and $u \in C \left( Q_1^+ \cup Q_1^*  \right)$ be bounded and satisfies in the viscosity sense
\begin{align*}
\begin{cases}
u \in S_p(\lambda,\Lambda,f), &\ \ \ \ \ \text{ in } \ Q_1^+ \\
u = 0, &\ \ \ \ \ \text{ on } \ Q_1^*. \\
\end{cases}
\end{align*}
Then there exist universal constants $0<\alpha<1, C>0$ so that for any $0< \rho \leq \frac{1}{2}$
\begin{equation}\label{osc_u/y_est}
\osc_{\overline{Q}_\rho^+} \frac{u}{y} \leq C \rho^\alpha \left( \osc_{\overline{Q}_{1/2}^+} \frac{u}{y} + ||f||_{L^\infty\left(  Q_{1}^+ \right)} \right).
\end{equation}
\end{lemma}

The proof can be found in the appendix. We continue with an immediate consequence. 

\begin{cor} \label{u_y}
Let $f$ be bounded in $Q_1^+$ and $u \in C \left( Q_1^+ \cup Q_1^*  \right)$ be bounded and satisfies 
\begin{align*}
\begin{cases}
u \in S_p(\lambda,\Lambda,f), &\ \ \ \ \ \text{ in } \ Q_1^+ \\
u = 0, &\ \ \ \ \ \text{ on } \ Q_1^*. \\
\end{cases}
\end{align*}
Then $u_y$ exists on $Q_1^*$ and for universal constants $C>0, 0<\alpha<1$ we have
\begin{equation} \label{u_y_est}
|u(X,t)-u_y(x,0,t)y| \leq C \left( \norm{u}_{L^\infty\left(\overline{Q}_1^+ \right)} + \norm{f}_{L^\infty\left(\overline{Q}_1^+ \right)} \right) y^{1+\alpha}
\end{equation}
for every $(X,t) \in \overline{Q}_{1/2}^+$. Moreover, $u_y$ is $H^\alpha\left(\overline{Q}_{1/2}^+ \right)$ with the corresponding norm depending only on universal quantities and $K:=\norm{u}_{L^\infty\left(\overline{Q}_1^+ \right)} + \norm{f}_{L^\infty\left(\overline{Q}_1^+ \right)} $.
\end{cor}

\begin{proof}
Note first that the justification for the existence and $H^\alpha$-regularity of $u_y$ can be found in the proof of Lemma \ref{Dirichlet_H1+a_hom} (see appendix). Next let $(X,t) \in \overline{Q}_{1/2}^+$. We apply Lemma \ref{osc_u/y}(rescaled) in $\overline{Q}_{y}^+(x,0,t) \subset Q_1^+$ to obtain for small $h>0$, $\frac{u(X,t)}{y}-\frac{u(x,h,t)}{h} \leq C K \ y^a.$ So letting $h \to 0$, $\frac{u(X,t)}{y}-u_y(x,0,t) \leq  C K \ y^a$.
\end{proof}

Next we apply the above to $u_y$ to obtain the following.

\begin{cor} \label{u_yy}
Let $u \in C \left( Q_1^+ \cup Q_1^*  \right)$ be bounded and satisfies in the viscosity sense 
\begin{align*}
\begin{cases}
F(D^2u)-u_t=0, &\ \ \ \ \ \text{ in } \ Q_1^+ \\
u_y = 0, &\ \ \ \ \ \text{ on } \ Q_1^*. \\
\end{cases}
\end{align*}
Then $u_{yy}$ exists on $Q_1^*$ and for a universal constants $C>0, 0<\alpha<1$ we have
\begin{equation} \label{u_yy_est}
\abs{u(X,t)-u(x,0,t)-\frac{1}{2}u_{yy}(x,0,t)\ y^2} \leq C  \norm{u}_{L^\infty\left(\overline{Q}_1^+ \right)} \ y^{2+\alpha}
\end{equation}
for every $(X,t) \in \overline{Q}_{1/2}^+$. Moreover, $u_{yy}$ is $H^\alpha\left(\overline{Q}_{1/2}^+ \right)$ with the corresponding norm depending only on universal quantities and $K:=\norm{u}_{L^\infty\left(\overline{Q}_1^+ \right)} $.
\end{cor}

\begin{proof}
First we observe that $u_y$ exists in $Q_1^+ \cup Q_1^*$ from Theorem \ref{Neumann_H1+a} and moreover it satisfies the following
\begin{align*}
\begin{cases}
u_y \in S_p\left( \frac{\lambda}{n},\Lambda\right), &\ \ \ \ \ \text{ in } \ Q_1^+ \\
u_y = 0, &\ \ \ \ \ \text{ on } \ Q_1^*. \\
\end{cases}
\end{align*}
Hence we can apply Corollary \ref{u_y} to $u_y$. This means that $u_{yy}$ exists and it is $H^\alpha\left(\overline{Q}_{1/2}^+ \right)$. Also from (\ref{u_y_est}) we have
$$-CKy^{1+\alpha} \leq u_y(X,t)-u_{yy}(x,0,t)y\leq CKy^{1+\alpha}$$
for any $(X,t) \in Q_{1/2}^+$. Then we integrate in direction $y$ and for any $(X,t) \in Q_{1/2}^+$ we obtain
\begin{align*}
u(X,t)-u(x,0,t) = \int_{0}^{y} u_y(x,\rho,t) \ d\rho &\leq \int_{0}^{y} \left( u_{yy}(x,0,t)\rho + CK\rho^{1+\alpha} \right) \ d\rho \\
&= u_{yy}(x,0,t) \frac{y^2}{2} + CK y^{2+\alpha}.
\end{align*}
\end{proof}

\begin{prop} \label{thin_equation}
Let $u \in C \left( Q_1^+ \cup Q_1^*  \right)$ be bounded and satisfies in the viscosity sense 
\begin{align*}
\begin{cases}
F(D^2u)-u_t=0, &\ \ \ \ \ \text{ in } \ Q_1^+ \\
u_y = 0, &\ \ \ \ \ \text{ on } \ Q_1^*. \\
\end{cases}
\end{align*}
Consider the restriction of $u$ on $Q_1^*$, $v(x,t):=u(x,0,t)$. Moreover, denoting by $A(x,t):=u_{yy}(x,0,t)$ (which exists regarding Corollary \ref{u_yy}) we consider the operator
\begin{align} \label{thin_equation_1}
G(M,x,t):=  F \left( \begin{matrix}
M  &  0\\
0 & A(x,t)
\end{matrix} \right)
\end{align}
for $(x,t) \in Q_1^*$ and $M \in S_{n-1}$. Then in the viscosity sense
$$G\left(D^2v,x,t\right) -v_t=0, \ \ \text{ in } \ \ Q_1^*.$$
\end{prop}

\begin{proof}
For convenience we show the result at $P_0=(0,0) \in Q_1^*$. Let $\phi$ be a test function on $Q_1^*$ that touches $v$ from below at $(0,0)$. Our aim is to show that
\begin{align*} 
F \left( \begin{matrix}
D^2\phi(0,0) &  0\\
0 & A(0,0)
\end{matrix} \right) - \phi_t(0,0) \leq 0.
\end{align*}
To do so we will try to extend $\phi$ into $Q_1^+$ and translate it suitably to turn it into a test function that touches $u$ at some point of $Q_r^+$. For small $\epsilon>0$ we consider, $\tilde{\phi}(X,t) = \phi(x,t) + \frac{A(0,0)}{2}y^2 - \epsilon (|X|^2-t).$ First, using Corollary \ref{u_yy} we can obtain that for sufficiently small $r>0$
\begin{equation} \label{thin_equation_2}
u(X,t) \geq \tilde{\phi}(X,t) +  \frac{\epsilon}{2}(|X|^2-t), \ \ \text{ for any } \ \ (X,t) \in  \overline{Q}_r^+.
\end{equation}
Indeed, Corollary \ref{u_yy} implies that for any $(X,t) \in Q_{1/2}^+$,
$$u(X,t) \geq u(x,0,t) + \frac{A(x,t)}{2}y^2 -CKy^{2+\alpha}$$
moreover, $A$ is $H^\alpha$ that is, $A(0,0)-A(x,t) \leq CK|x|^\alpha+CK|t|^{\frac{\alpha}{2}}.$
Hence $u(X,t)\geq u(x,0,t) + \frac{A(0,0)}{2}y^2 -CK|X|^{2+\alpha}-CK|t|^{\frac{\alpha}{2}} y^2$.
Now choose $0<r< \min \{ \rho, \left( \frac{\epsilon}{4CK} \right)^{1/\alpha} \}$, then for $(X,t) \in  \overline{Q}_r^+$ we have $u(X,t) \geq \phi(x,t) + \frac{A(0,0)}{2}y^2 -\frac{\epsilon}{2} (|X|^2-t)$.

Next, we translate suitably $\tilde{\phi}$ in order to achieve $u-\tilde{\phi}$ to have a local minimum. So we consider for $h \in \R$,
$$\tilde{\phi}_h(X,t) = \tilde{\phi}(x,y-h,t).$$
Then $\tilde{\phi}_h(X,t)= \tilde{\phi}(X,t) - A(0,0)yh + \frac{A(0,0)}{2}h^2 +2\epsilon hy- \epsilon h^2$.

Next, we observe that, $u(0,0,0)-\tilde{\phi}_h(0,0,0) = - \frac{A(0,0)}{2}h^2 + \epsilon h^2$
and by (\ref{thin_equation_2}),
\begin{equation*} \label{thin_equation_4}
u(X,t) - \tilde{\phi}_h(X,t) \geq  \frac{\epsilon}{2}(|X|^2-t) + (A(0,0)-2\epsilon)hy + u(0,0,0)-\tilde{\phi}_h(0,0,0)
\end{equation*}
for any $(X,t) \in \overline{Q}_r^+$. So we have the following
\begin{equation} \label{thin_equation_5}
u(X,t) - \tilde{\phi}_h(X,t) \geq  \frac{\epsilon}{2}r^2 + (A(0,0)-2\epsilon)hy + u(0,0,0)-\tilde{\phi}_h(0,0,0), \ \text{ on } \ \partial_p Q_r^+ \setminus Q_r^*.
\end{equation} 
\begin{equation} \label{thin_equation_6}
(\tilde{\phi}_h)_y = -A(0,0)h+2\epsilon h, \ \text{ on } \ \overline{Q}_r^*
\end{equation}

Subsequently, we split into two cases.

\hspace{-7mm} \underline{Case 1}: If $A(0,0) \leq 0$. We choose $h>0$ and we have: On $\partial_p Q_r^+ \setminus Q_r^*$, using (\ref{thin_equation_5}) we have, $u(X,t) - \tilde{\phi}_h(X,t) \geq  u(0,0,0)-\tilde{\phi}_h(0,0,0)$, choosing $0 < h \leq \frac{\epsilon r}{2(2\epsilon-A(0,0))}$. On $\overline{Q}_r^*$, by (\ref{thin_equation_6}) we know that $(\tilde{\phi}_h)_y>0$. Also $u_y=0$, hence $(u-\tilde{\phi}_h)_y <0$. This imply that  $u-\tilde{\phi}_h$ has a local (in the parabolic sense) minimum. Then, we use the equation at $(X_1,t_1)$, i.e. $F(D^2 \tilde{\phi}_h(X_1,t_1)) - (\tilde{\phi}_h)_t(X_1,t_1) \leq 0$. But
\begin{align*} 
D^2 \tilde{\phi}_h(X_1,t_1) = \left( \begin{matrix}
D^2 \phi (x_1,t_1) -2\epsilon I_{n-1} &  0\\
0 & A(0,0)-2\epsilon
\end{matrix} \right) 
\end{align*}
and, $(\tilde{\phi}_h)_t(X_1,t_1) =\phi_t(x_1,t_1) + \epsilon$. So, taking $\epsilon \to 0$ then $r \to 0$ and $(x_1,t_1) \to (0,0)$ and we obtain what we want.

\hspace{-7mm} \underline{Case 2}: If $A(0,0) >0$. We choose $h= -\bar{h}$, for $\bar{h}>0$ and $\epsilon < \frac{A(0,0)}{2}$, then we have: On $\partial_p Q_r^+ \setminus Q_r^*$, using (\ref{thin_equation_5}) we have, $u(X,t) - \tilde{\phi}_h(X,t) \geq  u(0,0,0)-\tilde{\phi}_h(0,0,0)$, choosing $0 < \bar{h} \leq \frac{\epsilon r}{2(A(0,0)-2\epsilon)}$. On $\overline{Q}_r^*$, by (\ref{thin_equation_6}) we have, $(\tilde{\phi}_h)_y = \bar{h}(A(0,0)-2\epsilon )>0$. Hence $(u-\tilde{\phi}_h)_y <0$. Then we can argue as in Case 1.

Finally note that a similar argument can be applied for test functions that touch $v$ by above.
\end{proof}

Now we are able to prove the main theorem of this section.

\begin{thm} \label{Neumann_H2+a} (Boundary $H^{2+\alpha}$-estimates for the Neumann problem).
Let $u \in C \left( Q_1^+ \cup Q_1^*  \right)$ be bounded and satisfy in the viscosity sense
\begin{align*} 
\begin{cases}
F(D^2u) -u_t = 0, &\ \ \ \ \ \text{ in } \ Q_1^+ \\
u_y = 0, &\ \ \ \ \ \text{ on } \ Q_1^*. \\
\end{cases}
\end{align*}
Then the second derivatives of $u$ exist in $\overline{Q}_{1/2}^+$. Moreover there exists universal constant $0<\alpha < 1$ and a polynomial $R_{2;P_0}(X,t)= A_{P_0}+ B_{P_0} \cdot (X-X_0)+C_{P_0} (t-t_0) + \frac{1}{2} (X-X_0)^\tau D_{P_0} (X-X_0)$, where $A_{P_0}=u(P_0), B_{P_0}= \left( u_{x_1}(P_0), \dots, u_{x_{n-1}}(P_0), 0 \right), C_{P_0}=u_t(P_0)$ and 
\begin{align*} 
D_{P_0}:=   \left( \begin{matrix}
u_{x_1x_1}(P_0)\ \dots \ u_{x_1x_{n-1}}(P_0)  &  0\\
\vdots\ \ \ \ \ \ \ddots \ \ \ \ \ \ \vdots &\vdots\\
u_{x_{n-1}x_1}(P_0)\  \dots \ u_{x_{n-1}x_{n-1}}(P_0)  &  0\\
0 \ \ \ \ \ \ \dots \ \ \ \ \ \ 0 & u_{yy}(P_0)
\end{matrix} \right)
\end{align*}
for $P_0 \in Q_{1/2}^*$, so that 
\begin{equation}\label{Neumann_H2+a_est}
|u(X,t) - R_{2;P_0}(X,t)| \leq C \left( ||u||_{L^\infty \left( Q_{1}^+ \right)}+|F(O)|\right) \ p(P,P_0)^{2+\alpha}
\end{equation}
for every $P=(X,t) \in \overline{Q}_{1/2}^+(P_0)$, where $C>0$ is a universal constant.
\end{thm}

Note that in this case the existence and $H^\alpha$-regulatiy of $u_t$ is already known from Theorem \ref{Neumann_H1+a}.

\begin{proof}
Our intention is to combine Corollary \ref{u_yy} and interior $H^{2+\alpha}$-estimates on $Q_1^*$ once from Proposition \ref{thin_equation} $u$ satisfies an equation there.

So, let $v(x,t)=u(x,0,t)$. Then $v$ satisfies $G(D^2v(x,t),(x,t))-v_t(x,t)=0$ in $Q_1^*$, where $G$ is defined in (\ref{thin_equation_1}). In order to use interior $H^{2+\alpha}$-estimates we have to verify that this equation satisfies the assumptions of Theorem \ref{interior_H2+a_2}. It is easy to check that $G$ has the same ellipticity constants as $F$. Next we examine if the quantity $\theta_G$ satisfies the assumption (\ref{interior_H2+a_2_est2}). Since $F$ is Lipschitz we have
\begin{align*}
0 \leq \theta_G(x,t) \leq  \sup_{M \in S_{n-1}} \frac{\norm{  \left( \begin{matrix}
0  &  0\\
0 & A(x,t)-A(0,0)
\end{matrix} \right) }}{|M|+1} \leq CK \max \{ |x|,|t|^{1/2} \}^\alpha.
\end{align*}
Finally, the assumption (\ref{interior_H2+a_2_est1}) can be derived by interior $H^{2+\alpha}$-estimates observing that the operator $G(M+B,(0,0))-E$ is convex and has the same ellipticity constants as $G$. 

We will show the result at $P_0=(0,0,0)$, for convenience. Applying Theorem \ref{interior_H2+a_2} to $v$ and we obtain that there exists a polynomial $\tilde{R}_{2;P_0}(x,t)= \tilde{A}_{P_0}+ \tilde{B}_{P_0} \cdot x+\tilde{C}_{P_0} t + \frac{1}{2} x^\tau \tilde{D}_{P_0} x$, where $\tilde{A}_{P_0}=v(0,0), \tilde{B}_{P_0}= \nabla_x v(0,0), \tilde{C}_{P_0}=v_t(0,0)$ and $\tilde{D}_{P_0}:=  D^2_x v(0,0)$ so that
\begin{equation*}
|u(x,0,t) - \tilde{R}_{2;P_0}(x,t)| \leq C \left( ||u||_{L^\infty \left( Q_{1}^+ \right)}+|F(O)|\right) \ \max \{ |x|,|t|^{1/2} \}^{2+\alpha}
\end{equation*}
for every $(x,t) \in \overline{Q}^*_{1/2}$. On the other hand we have estimate (\ref{u_yy_est}) of Corollary \ref{u_yy} which gives for $(X,t) \in \overline{Q}^+_{1/2}$,
\begin{align*}
\abs{u(X,t)-u(x,0,t)-\frac{1}{2}A(0,0)\ y^2} &\leq C  \norm{u}_{L^\infty\left(\overline{Q}_1^+ \right)} \ y^{2+\alpha} + \abs{A(x,t)-A(0,0)} \frac{y^2}{2} \\
&\leq C\left(||u||_{L^\infty \left( Q_{1}^+ \right)}+|F(O)|\right) \ \max \{ |X|,|t|^{1/2} \}^{2+\alpha}
\end{align*}
Then, we take $R_{2;P_0}(X,t) = \tilde{R}_{2;P_0}(x,t) + \frac{A(0,0)}{2}y^2$ and we get the result.
\end{proof}

\subsection{$H^{2+\alpha}$-estimates for the oblique derivative case}
In the present section we intent to obtain $H^{2+\alpha}$-estimates for the general oblique derivative problem (Theorem \ref{oblique_general_H2+a}). We achieve this again using an approximation method. We "approximate" the general problem by homogeneous problems with a suitable function $\beta$ in the oblique derivative condition (as in Lemma \ref{oblique_general_H2+a_ii}). To get Lemma \ref{oblique_general_H2+a_ii} we need to examine first the case when we have a non-homogeneous oblique derivative condition but with constant $\beta$ (Lemma \ref{oblique_general_H2+a_i}) which can be done again by approximating the problem with suitable constant oblique derivative problems. Thereafter we first examine a constant oblique derivative problem (Theorem \ref{oblique_constant_H2+a}) using the change of variables of section \ref{changeofvar}. For convenience we assume that $F(O)=0$.

\begin{thm} \label{oblique_constant_H2+a} (Boundary $H^{2+\alpha}$-estimates for the constant oblique derivative problem).
Let $u \in C \left( Q_1^+ \cup Q_1^*  \right)$ be bounded and satisfy in the viscosity sense
\begin{align*} 
\begin{cases}
F(D^2u) -u_t = 0, &\ \ \ \ \ \text{ in } \ Q_1^+ \\
\beta \cdot Du = 0, &\ \ \ \ \ \text{ on } \ Q_1^*\\
\end{cases}
\end{align*}
where $\beta$ is a constant function. Then the second derivatives of $u$ exist at $(0,0)$. Moreover there exists universal constant $0<\alpha < 1$ and a polynomial $R_{2;0}(Z,t)= A^{0}+ B^{0} \cdot Z+C^{0} t + \frac{1}{2} Z^\tau D^{0} Z$, where $A^{0}=u(0,0), B^{0}=Du(0,0) \in \R^n, C^{0}=u_t(0,0)$ and $D^0=D^2u(0,0) \in S_n$ so that
\begin{equation}\label{oblique_constant_H2+a_est}
|u(Z,t) - R_{2;0}(Z,t)| \leq C \ ||u||_{L^\infty \left( Q_{1}^+ \right)} \ \left(|Z|+|t|^{1/2}\right)^{2+\alpha}
\end{equation}
for every $(Z,t) \in \overline{Q}_{\rho}^+$, where $C>0$ and $0<\rho<1$ are universal constants.
\end{thm}
 
\begin{proof}
Let $A$ be the transformation defined in section \ref{changeofvar}. Define $v(X,t)=u(AX,t)$, for $(X,t) \in Q_r^+$, where $0<r<\frac{\delta_0}{\delta_0+1}<1$. Note that $Q_r^+ \subset \tilde{Q_1^+}$. Then 
\begin{align*} 
\begin{cases}
\tilde{F}(D^2v) -v_t = 0, &\ \ \ \ \ \text{ in } \ \tilde{Q_r^+} \\
v_y = 0, &\ \ \ \ \ \text{ on } \ Q_r^*\\
\end{cases}
\end{align*}
with $\tilde{F}$ convex. So applying Theorem \ref{Neumann_H2+a} to $v$ we have that the second derivatives of $v$ exist at $(0,0)$ and there exists a polynomial $\tilde{R}_{2}(X,t)=\tilde{A}^0+\tilde{B}^0 \cdot X+\tilde{C}^0 t + \frac{1}{2} X^\tau \tilde{D}^{0} X$, where $\tilde{A}^{0}=v(0,0), \tilde{B}^{0}=Dv(0,0) \in \R^n, \tilde{C}^0=v_t(0,0)$ and $\tilde{D}^0=D^2v(0,0) \in S_n$  so that 
$$|v(X,t) - \tilde{R}_2(X,t)| \leq C \ ||v||_{L^\infty \left( Q_{r}^+ \right)} \ \left(|X|+|t|^{1/2}\right)^{2+\alpha}$$
for every $(X,t) \in \overline{Q}_{r/2}^+$, where $C>0$, $0<\alpha<1$ are universal constants.

Let $R_2(Z,t)=\tilde{R}_2(A^{-1}Z,t)=\ \tilde{A}^0+\tilde{B}^0\cdot A^{-1}Z+\tilde{C}^0 t + \frac{1}{2} (A^{-1}Z)^\tau \tilde{D}^{0} A^{-1}Z =\ \tilde{A}^0+(A^{-1})^\tau \tilde{B}^0\cdot Z+\tilde{C}^0 t + \frac{1}{2} Z^\tau (A^{-1})^\tau \tilde{D}^{0} A^{-1}Z$ and observe that $\tilde{A}^0=v(0,0)=u(0,0)=:A^0$
and
\begin{align*}
(A^{-1})^\tau \tilde{B}^0&=\left( v_{x_1}(0,0),\dots,v_{x_{n-1}}(0,0), v_y(0,0)- \frac{\beta _1}{\beta_n}v_{x_1}(0,0)- \dots - \frac{\beta _{n-1}}{\beta_n}v_{x_{n-1}}(0,0)  \right) \\ 
&=\left( u_{z_1}(0,0),\dots,u_{z_{n-1}}(0,0), u_w(0,0)  \right)=:B^0
\end{align*}
and $\tilde{C}^0=v_t(0,0)=u_t(0,0)=:C^0$, $(A^{-1})^\tau \tilde{D}^{0} A^{-1}=D^2u(0,0)=:D^0$. Then
$$|u(Z,t) - R_2(Z,t)| \leq C \ ||u||_{L^\infty \left( Q_{r}^+ \right)} \ \left(|Z|+|t|^{1/2}\right)^{2+\alpha}$$
for every $(Z,t) \in \overline{Q}_{\rho}^+$, for $\rho=\frac{\delta_0r}{2(\delta_0+1)}$ . 
\end{proof} 

\begin{thm} \label{oblique_general_H2+a} (Boundary $H^{2+\alpha}$-estimates for the general oblique derivative problem).
Let $g$ and $\beta$ be $H^{1+\gamma}$ locally on  $Q_{1}^*$, $f \in H^\gamma \left( Q_1^+\right)$ and $u \in C \left( Q_1^+ \cup Q_1^*  \right)$ be bounded and satisfies in the viscosity sense
\begin{align*}
\begin{cases}
F(D^2u) -u_t = f, &\ \ \ \ \ \text{ in } \ Q_1^+ \\
\beta \cdot Du = g, &\ \ \ \ \ \text{ on } \ Q_1^*. \\
\end{cases}
\end{align*}
Then the second derivatives of $u$ and $u_t$ exist at $(0,0)$. Moreover there exists universal constant $0<\alpha_0 < 1$ and a polynomial $R_{2;0}(X,t)= A^{0}+ B^{0} \cdot X+
\Gamma^{0} t + \frac{1}{2} X^\tau D^{0} X$, where $A^{0}=u(0,0), B^{0}= Du(0,0) \in \R^n,  \Gamma^{0}=u_t(0,0)$ and $D^0=D^2u(0,0) \in S_n$ so that for $\alpha=\min\{\alpha_0, \gamma\}$,
\begin{equation}\label{oblique_general_H2+a_est}
|u(X,t) - R_{2;0}(X,t)| \leq C \left( ||u||_{L^\infty \left( Q_{1}^+ \right)}+||g||_{H^{1+\gamma}\left(  \overline{Q}_{1/2}^* \right)} +||f||_{ H^\gamma \left( Q_1^+\right)} \right) \ \left(|X|+|t|^{1/2}\right)^{2+\alpha}
\end{equation}
for every $(X,t) \in \overline{Q}_{1/4}^+$, where $C>0$ is a universal constant.
\end{thm}

Note that we may assume that: $u(0,0)=0$ and $g(0,0)=0$.$f(0,0)=0$, considering $F'(M):=F(M)-f(0,0)$, then $F'(D^2u)-u_t=f-f(0,0)$. $g_{x_i}(0,0)=0$ for every $i=1,\dots,n-1$, considering 
$$\bar{u}(X,t):= u(X,t) - \frac{y}{\beta_n(0,0)} \sum_{k=1}^{n-1} g_{x_k}(0,0)x_k.$$
Then for
\begin{align*}
M_0= \left(  \begin{matrix}
0& \dots & 0 & \frac{g_{x_1}(0,0)}{\beta_n(0,0)}\\
\vdots & \vdots & \vdots &\vdots\\
0 & \dots & 0 & \frac{g_{x_{n-1}}(0,0)}{\beta_n(0,0)}\\
\frac{g_{x_1}(0,0)}{\beta_n(0,0)} & \dots & \frac{g_{x_{n-1}}(0,0)}{\beta_n(0,0)} & 0
\end{matrix} \right) \in S_n,
\end{align*} 
$F\left( D^2 \bar{u} + M_0 \right) - \bar{u}_t=f$ in $Q_1^+$ and $\bar{F}(M):=F(M+M_0)$ has the same ellipticity constants as $F$.

The next remark will be useful in the following proofs.

\begin{preremark} \label{useful_matrix}
Let 
\begin{align*}
D= \left(  \begin{matrix}
0& \dots & 0 & 0\\
\vdots & \vdots & \vdots &\vdots\\
0 & \dots & 0 & 0\\
0 & \dots & 0 & 1
\end{matrix} \right) 
\end{align*} 
Then there exists $t_0 \in \R$ so that $\bar{F}(\tau_0D)=0$. Moreover, $\abs{\tau_0} \leq C ||g||_{H^{1+\gamma}\left(  \overline{Q}_{1/2}^* \right)}$, where $C>0$ universal.
\end{preremark}

Indeed, denoting by $l:=\frac{|F(M_0)|}{\lambda}$, the ellipticity conditions gives that $ F \left( M_0+lD \right)  \geq 0$ and 
$F \left(M_0-lD \right)  \leq 0$.

Note that, in the following we denote $\bar{u}, \bar{g}, \bar{F}$ by $u, g, F$ for convenience. As we mention in the start, in order to prove Theorem \ref{oblique_general_H2+a} we prove first two special cases.

\begin{lemma} \label{oblique_general_H2+a_i} 
We assume the same as in Theorem \ref{oblique_general_H2+a} but with $f=0$ and $\beta$ a constant vector. Then the second derivatives of $u$ and $u_t$ exist at $(0,0)$. Moreover there exists universal constant $0<\alpha_0 < 1$ and a polynomial $R_{2;0}(X,t)= A^{0}+ B^{0} \cdot X+\Gamma^{0} t + \frac{1}{2} X^\tau D^{0} X$, where $A^{0}=u(0,0), B^{0}= Du(0,0) \in \R^n,  \Gamma^{0}=u_t(0,0)$ and $D^0=D^2u(0,0) \in S_n$ so that for $\alpha=\min\{\alpha_0, \gamma\}$,
\begin{equation}\label{oblique_general_H2+a_i_est}
|u(X,t) - R_{2;0}(X,t)| \leq C \left( ||u||_{L^\infty \left( Q_{1}^+ \right)}+||g||_{H^{1+\gamma}\left(  \overline{Q}_{1/2}^* \right)}  \right) \ \left(|X|+|t|^{1/2}\right)^{2+\alpha}
\end{equation}
for every $(X,t) \in \overline{Q}_{1/4}^+$, where $C>0$ is a universal constant.
\end{lemma}

\begin{proof}
Before we start let us denote for convenience $K:=||u||_{L^\infty \left( Q_{1}^+ \right)}+||g||_{H^{1+\gamma}\left(  \overline{Q}_{1/2}^* \right)}$.

We intend to find some $R^{0}(X,t)= B^{0} \cdot X+\Gamma^{0} t + \frac{1}{2} X^\tau D^{0} X$, with  $\beta \cdot B^0 =0$ and $F(D^0)-\Gamma^0=0$ so that for universal $C>0, 0<\eta<1, \alpha_0>0$ and $\alpha = \min\{\alpha_0, \gamma\}$ we will have
\begin{equation} \label{oblique_general_H2+a_i_est1} 
\osc_{Q_{\eta^k}^+} \left( u(X,t) - R^{0}(X,t) \right) \leq CK \eta^{k(2+\alpha)}, \ \ \text{ for any } \ \ k \in \N. 
\end{equation}

Now, to prove (\ref{oblique_general_H2+a_i_est1}) we are going to show by induction that there exist universal constants $0<\eta <<1, \bar{C}>0, \alpha_0>0$ such that for $\alpha = \min\{\alpha_0, \gamma\}$ we can find a paraboloid $R_k(X,t)= B_k\cdot X+\Gamma_k t + \frac{1}{2} X^\tau D_k X$, with 
\begin{equation}\label{oblique_general_H2+a_i_est2_0} 
F(D_k)-\Gamma_k=0, \ \ \beta \cdot B_k=0 \ \ \text{ and }  \ \ \sum_{j=1}^n(D_k)_{ij} \beta_j=0, \ i=1,\dots,n-1
\end{equation}
for any $k \in \N$ so that
\begin{equation} \label{oblique_general_H2+a_i_est2} 
\osc_{Q_{\eta^k}^+} \left( u(X,t) - R_k(X,t) \right) \leq \bar{C}K \eta^{k(2+\alpha)}
\end{equation}
and
\begin{equation} \label{oblique_general_H2+a_i_est3} 
||D_{k+1}-D_k|| \leq CK\eta^{k\alpha}, \ |\Gamma_{k+1}-\Gamma_k|| \leq CK\eta^{k\alpha}, \ |B_{k+1}-B_k| \leq CK\eta^{k(1+\alpha)}.
\end{equation}

First, for $k=0$, take $B_0=0$, $\Gamma_0=0$ and $(D_0)_{ij}=0$, for $ij \neq nn$ and $(D_0)_{nn}=\tau_0$ where $\tau_0$ is chosen so that $F(D_0)=0$ (see Remark \ref{useful_matrix}) and $\bar{C}$ large enough. 

Next for the induction we assume that we have found paraboloids $R_0,R_1,\dots,R_{k_0}$ for which  (\ref{oblique_general_H2+a_i_est2_0}), (\ref{oblique_general_H2+a_i_est2}) and (\ref{oblique_general_H2+a_i_est3}) are true. Denoting by $r:=\eta ^{k_0}$ we have 
\begin{equation}\label{oblique_general_H2+a_i_est4} 
\osc_{Q_r^+} \left( u(X,t) - R_{k_0}(X,t) \right) \leq \bar{C}K r^{(2+\alpha)}.
\end{equation}

Now we are going to consider a suitable constant oblique derivative problem (as the one of Theorem \ref{oblique_constant_H2+a}). So let $v$ be the viscosity solution of
\begin{align*}
\begin{cases}
G(D^2v) -v_t = 0, &\ \ \ \ \ \text{ in } \ Q_r^+ \\
\beta \cdot Dv=0, &\ \ \ \ \ \text{ on } \ Q_r^* \\
v=u-R_{k_0}, &\ \ \ \ \ \text{ on } \ \partial_p Q_r^+ \setminus Q_r^*
\end{cases}
\end{align*}
where $G(M)=F(M+D_{k_0})-\Gamma_{k_0}$ which is an elliptic operator with the same ellipticity constants as $F$. Also $G(O)=F(D_{k_0})-\Gamma_{k_0}=0$.
Then $v$ satisfies ABPT-estimate for the oblique derivative case (see Theorem \ref{ABP-Oblique}) which gives
\begin{equation} \label{oblique_general_H2+a_i_est6} 
\osc_{Q_{r}^+} v \leq \osc_{Q_{r}^+}\left( u(X,t) - R_{k_0}(X,t) \right).
\end{equation}
From Theorem \ref{oblique_constant_H2+a} we have that  $\bar{B}:=Dv(0,0)$, $\bar{\Gamma}:=v_t(0,0)$, $\bar{D}:=D^2v(0,0)$ exist and for $\bar{R}(X,t)=\bar{B}\cdot X + \bar{\Gamma}t+\frac{1}{2}X^\tau \bar{D}X$ we have
\begin{equation} \label{oblique_general_H2+a_i_est7} 
\osc_{Q_{\tilde{r}}^+} \left(v(X,t)-\bar{R} (X,t) \right) \leq C_0 \left( \frac{\tilde{r}}{r}\right)^{2+\alpha_1} \ \osc_{Q_{r}^+} v 
\end{equation}
for any $\tilde{r} \leq \rho \ r$, where $0<\rho<1$ universal and also 
\begin{equation} \label{oblique_general_H2+a_i_est8} 
|\bar{B}| \leq  \frac{C}{r}\osc_{Q_{r}^+} v, \ |\bar{\Gamma}| \leq  \frac{C}{r^2}\osc_{Q_{r}^+} v, \ ||\bar{D}||_\infty \leq  \frac{C}{r^2}\osc_{Q_{r}^+} v.  
\end{equation}
Note that $\beta \cdot \bar{B}=0$ and $F(\bar{D}+D_{k_0})-\Gamma_{k_0}-\bar{\Gamma}=0$. Also, $\beta \cdot Dv=0$ holds in the classical sense on $Q_1^*$ and we can differentiate this condition with respect to $x_i$, $i\leq n-1$ to get $\sum_{j=1}^n \bar{D}_{ij} \beta_j=0$.

Next, we take $\tilde{r}=\eta r$ (for $0<\eta<\rho$) in (\ref{oblique_general_H2+a_i_est7} ). Hence
\begin{equation} \label{oblique_general_H2+a_i_est9} 
\osc_{Q_{\eta r}^+} \left(v(X,t)- \bar{R} (X,t) \right) \leq C_0 \eta^{2+\alpha_1} \osc_{Q_{r}^+} v.
\end{equation}
Now take (universal) $0<\eta<<1$ sufficiently small in order to have that $C_0 \eta^{\alpha_1}<1$. We denote by $1-\theta :=C_0 \eta^{\alpha_1}$, where $0<\theta<1$ is a universal constant. Then 
\begin{equation} \label{oblique_general_H2+a_i_est10} 
\osc_{Q_{\eta r}^+} \left(v(X,t)-\bar{R}(X,t) \right) \leq (1-\theta) \ \eta^2 \ \osc_{Q_{r}^+}\left( u(X,t) - R_{k_0}  (X,t) \right) \leq (1-\theta)\ \eta^2 \ \bar{C}Kr^{2+\alpha}.
\end{equation}

Now to return to $u$ we define $w=u-R_{k_0}-v$. Note that $F(D^2(R_{k_0}+v))-(R_{k_0}+v)_t=F(D_{k_0}+D^2v)-\Gamma_{k_0}-v_t=0$. Moreover we can easily check that $DR_{k_0} = D_{k_0}X+B_{k_0}$, then on $Q_r^*$, $\beta \cdot DR_{k_0}= \sum_{k=1}^{n-1}\sum_{j=1}^n \beta_j  (D_{k_0})_{jk}x_k=0$. That is combining the above we have
\begin{align*}
\begin{cases}
w \in S_p\left( \frac{\lambda}{n}, \Lambda  \right), &\ \ \ \ \ \text{ in } \ Q_r^+ \\
\beta \cdot Dw =g, &\ \ \ \ \ \text{ on } \ Q_r^* \\
w=0, &\ \ \ \ \ \text{ on } \ \partial_p Q_r^+ \setminus Q_r^*.
\end{cases}
\end{align*}

Next we apply again Theorem \ref{ABP-Oblique} and then the $H^{1+\gamma}$-estimate for $g$ together with the fact that $g(0,0)=0$ and $Dg(0,0)=0$ to obtain
\begin{align*}\label{oblique_general_H2+a_i_est11}
\osc_{Q_{\bar{r}}^+} w \leq Cr ||g||_{L^\infty(Q^*_r)} =  Cr ||g-g(0,0)-Dg(0,0)||_{L^\infty(Q^*_r)}\leq CKr^{2+\gamma}.
\end{align*}

Next combining the above with (\ref{oblique_general_H2+a_i_est9}) we get
\begin{equation}\label{oblique_general_H2+a_i_est12} 
\osc_{Q_{ \eta r}^+} \left[ u(X,t)-(R_{k_0}(X,t)+\bar{R}(X,t))  \right] \leq (1-\theta) \eta^2 \bar{C}Kr^{2+\alpha} +C K r^{2+\gamma }.
\end{equation}
We choose the right constants $\alpha_0$ and $\bar{C}$. So, take $\alpha_0$ so that $\eta^{\alpha_0}=1-\frac{\theta}{2}$ and $\alpha=\min\{\alpha_0, \gamma\}$ and $\bar{C}$ large enough so that $\frac{\eta ^2\theta \bar{C}}{2} \geq C$. Then we return to (\ref{oblique_general_H2+a_i_est12}) writing $1-\theta$ as $1-\frac{\theta}{2}-\frac{\theta}{2}$ and recalling that $r=\eta^{k_0}$,
\begin{align*}
\osc_{Q_{ \eta^{k_0+1}}^+} \left[ u(X,t)-(R_{k_0}(X,t)+\bar{R}(X,t)) \right] &\leq K \left[\left(1-\frac{\theta}{2}\right)\bar{C}\eta^2 \eta^{k_0(2+\alpha)} +C r^{2+\alpha } -\frac{\eta^2 \theta \bar{C}}{2} r^{2+\alpha} \right]\\
&\leq \bar{C}K \eta^{(k_0+1)(2+\alpha)}.
\end{align*}
Choosing $R_{k_0+1}=R_{k_0}+\bar{R}$ we have (\ref{oblique_general_H2+a_i_est2}) for $k_0+1$. Note also that $F(D_{k_0}+\bar{D})-(\Gamma_{k_0}+\bar{\Gamma})=0$, $\beta \cdot B_{k_0+1} = 0$ and for any $i\leq n-1$, $\sum_{j=1}^n (D_{k_0+1})_{ij}\beta_j=0$. It remains to get (\ref{oblique_general_H2+a_i_est3}) for $k=k_0$. To do so, we use relation (\ref{oblique_general_H2+a_i_est8}) together with (\ref{oblique_general_H2+a_i_est6}) and then (\ref{oblique_general_H2+a_i_est4}).

Finally, it remains to get estimate (\ref{oblique_general_H2+a_i_est1}). Observe that (\ref{oblique_general_H2+a_i_est3}) yields the existence of the limits $B_{\infty}:= \lim_{k \to \infty} B_k$, $\Gamma_{\infty}:= \lim_{k\to \infty} \Gamma_k$ and $D_{\infty}:= \lim_{k\to \infty} D_k$ exist and $R^0(X,t)=B_{\infty} \cdot X + \Gamma_{\infty}t + \frac{1}{2}X^\tau D_{\infty}X$ satisfies (\ref{oblique_general_H2+a_i_est1}). Indeed, $\beta \cdot B_\infty=0$, $F(D_\infty)-\Gamma_{\infty}=0$ and for any $k \in \N$ we have
\begin{align*}
\osc_{Q_{\eta^k}^+} \left( u(X,t)-R^0(X,t) \right)&\leq \osc_{Q_{\eta^k}^+} \left( u(X,t)-R^0(X,t) \right)\\
&\ \ \ + \eta^k |B_k-B_\infty|+\eta^{2k} |\Gamma_k-\Gamma_\infty|+\frac{1}{2}\eta^{2k} ||D_k-D_\infty|| \\
&\leq \bar{C}K \eta^{k(2+\alpha)} + CK \eta^{k} \sum_{j=k}^\infty \eta^{j(1+\alpha)}+2CK \eta^{2k} \sum_{j=k}^\infty \eta^{j\alpha} \\
&\leq CK \eta^{k(2+\alpha)}
\end{align*}
using the sum of geometric series.
\end{proof}

\begin{lemma} \label{oblique_general_H2+a_ii} 
Let $F$ be convex, $\beta$ be constant function, $N_0 \in \R^{n \times n}$ with $||N_0||_\infty \leq C_1$ and $u \in C \left( Q_1^+ \cup Q_1^*  \right)$ be bounded and satisfies in the viscosity sense
\begin{align*}
\begin{cases}
F(D^2u) -u_t = 0, &\ \ \ \ \ \text{ in } \ Q_1^+ \\
\left(\beta+N_0X \right) \cdot Du = 0, &\ \ \ \ \ \text{ on } \ Q_1^*. \\
\end{cases}
\end{align*}
Then the second derivatives of $u$ and $u_t$ exist at $(0,0)$. Moreover there exists universal constant $0<\alpha < 1$ and a polynomial $R_{2;0}(X,t)= A^{0}+ B^{0} \cdot X+\Gamma^{0} t + \frac{1}{2} X^\tau D^{0} X$, where $A^{0}=u(0,0), B^{0}= Du(0,0) \in \R^n,  \Gamma^{0}=u_t(0,0)$ and $D^0=D^2u(0,0) \in S_n$
so that
\begin{equation}\label{oblique_general_H2+a_ii_est}
|u(X,t) - R_{2;0}(X,t)| \leq C  ||u||_{L^\infty \left( Q_{1}^+ \right)} \ \left(|X|+|t|^{1/2}\right)^{2+\alpha}
\end{equation}
for every $(X,t) \in \overline{Q}_{1/4}^+$, where $C>0$ depends on universal constants and on $C_1$.
\end{lemma}

\begin{proof}
Our intention here is to "convert" our problem into a constant non-homogeneous oblique derivative problem in order to use the result of Lemma \ref{oblique_general_H2+a_i}. To do so we add to $u$ a suitable paraboloid. Note that $u$ satisfies $H^{1+\alpha}$-estimates locally up to the flat boundary and $H^{2+\alpha}$-interior estimates so it is in fact a classical solution.

First we choose $N \in S_n$ so that $N\beta=A^\tau Du(0,0)$. Note that such a matrix exists since the above is actually a linear system of $n$ equations and $\frac{n(n+1)}{2}$ variables and the matrix of the system can be shown to have rank equals to $n$ (using that $\beta_n \neq 0$). Moreover $||N||_\infty \leq C(n,\delta_0) |Du(0,0)|$. Then we define $v(X,t):= u(X,t) +\frac{1}{2} X^\tau N X$. Then $F(D^2v-N)-v_t=0,  \ \ \text{ in } \ \ Q_1^+.$ Also, for $X \in Q_1^*$, 
\begin{align*}
\beta \cdot Dv(X,t) &= -N_0X\cdot Du(X,t) + \beta \cdot NX = -(Du(X,t))^\tau N_0X+X^\tau N_0^\tau Du(0,0)\\
&=(Du(0,0)-Du(x,0,t))^\tau N_0 (x,0)=:g(x,t).
\end{align*}
We observe also that $v(0,0)=u(0,0)=0$, $G(M):=F(M-N)$ has the same ellipticity constants as $F$ and $||g||_{L^\infty \left(Q_r^*\right)} \leq ||N_0||_\infty r  ||Du(0,0)-Du(x,0,t)||_{L^\infty \left(Q_r^*\right)}\leq C ||u||_{L^\infty \left(Q_1^+\right)} r^{1+\alpha}$.

Therefore we can apply Lemma \ref{oblique_general_H2+a_i} to $v$ to obtain that there exists $\bar{R}(X,t)=\bar{B}\cdot X +\bar{\Gamma}t+\frac{1}{2} X^\tau \bar{D} X$ so that
$$||v-\bar{R}||_{L^\infty \left(Q_r^+\right)} \leq C \left( ||v||_{L^\infty \left(Q_1^+\right)}+||g||_{H^{1+\alpha} \left(Q_{1/2}^+\right)}+|F(-N)| \right) r^{2+\alpha} \leq C  ||u||_{L^\infty \left(Q_1^+\right)} r^{2+\alpha}$$
for any $r \leq \frac{1}{4}$. Taking as $R^0(X,t):=\bar{R}(X,t)+\frac{1}{2}X^\tau N X$ the proof is complete.
\end{proof}

\begin{proof}[Proof of Theorem \ref{oblique_general_H2+a}]
Before we start let us denote for convenience $K:=||u||_{L^\infty \left( Q_{1}^+ \right)}+||g||_{H^{1+\gamma}\left(  \overline{Q}_{1/2}^* \right)} +||f||_{ H^\gamma \left( Q_1^+\right)} $ and $\beta^0:=\beta (0,0),\ \beta_{x_i}^0:=\beta_{x_i} (0,0) \in \R^n$.

We intend to find some $R^{0}(X,t)= B^{0} \cdot X+\Gamma^{0} t + \frac{1}{2} X^\tau D^{0} X$, with  $\beta^0 \cdot B^0 =0$ and $F(D^0)-\Gamma^0=0$ so that for universal $C>0, 0<\eta<1, 0<\rho<1, \alpha_0>0$ and $\alpha = \min\{\alpha_0, \gamma\}$ we will have
\begin{equation} \label{oblique_general_H2+a_est1} 
\osc_{Q_{\rho \eta^k}^+} \left( u(X,t) - R^{0}(X,t) \right) \leq CK \eta^{k(2+\alpha)}, \ \ \text{ for any } \ \ k \in \N. 
\end{equation}

Now, to prove (\ref{oblique_general_H2+a_est1}) we are going to show by induction that there exist universal constants $0<\eta <<1, 0<\rho<1, \bar{C}>0, \alpha_0>0$ such that for $\alpha = \min\{\alpha_0, \gamma\}$ we can find a paraboloid $R_k(X,t)= B_k\cdot X+\Gamma_k t + \frac{1}{2} X^\tau D_k X$, with 
\begin{align}\label{oblique_general_H2+a_est2_0}
F(D_k)-\Gamma_k=0, \ \ \beta^0 \cdot B_k=0 \ \ \text{ and }\ \sum_{j=1}^n \left[ (D_k)_{ij} \beta_j^0+(\beta_j)_{x_i}^0 (B_k)_j\right]=0, \ i\leq n-1
\end{align}
for any $k \in \N$ so that
\begin{equation} \label{oblique_general_H2+a_est2} 
\osc_{Q_{\rho \eta^k}^+} \left( u(X,t) - R_k(X,t) \right) \leq \bar{C}K \eta^{k(2+\alpha)}
\end{equation}
and
\begin{equation} \label{oblique_general_H2+a_est3} 
||D_{k+1}-D_k|| \leq CK\eta^{k\alpha}, \ |\Gamma_{k+1}-\Gamma_k|| \leq CK\eta^{k\alpha}, \ |B_{k+1}-B_k| \leq CK\eta^{k(1+\alpha)}.
\end{equation}

First, for $k=0$, take $B_0=0$, $\Gamma_0=0$ and $(D_0)_{ij}=0$, for $ij \neq nn$ and $(D_0)_{nn}=\tau_0$ where $\tau_0$ is chosen so that $F(D_0)=0$ (see Remark \ref{useful_matrix}) and $\bar{C}$ large enough. 

Next for the induction we assume that we have found paraboloids $R_0,R_1,\dots,R_{k_0}$ for which  (\ref{oblique_general_H2+a_est2_0}), (\ref{oblique_general_H2+a_est2}) and (\ref{oblique_general_H2+a_est3}) are true. Denoting by $r:=\frac{\rho \eta ^{k_0}}{2}$ we have 
\begin{equation}\label{oblique_general_H2+a_est4} 
\osc_{Q_r^+} \left( u(X,t) - R_{k_0}(X,t) \right) \leq \frac{4}{\rho^{2+\alpha}} \bar{C}K r^{(2+\alpha)}.
\end{equation}

Now we are going to consider a suitable oblique derivative problem (as the one of Lemma \ref{oblique_general_H2+a_ii}). So let $v$ be the viscosity solution of
\begin{align*}
\begin{cases}
G(D^2v) -v_t = 0, &\ \ \ \ \ \text{ in } \ Q_r^+ \\
(\beta^0+D\beta^0x) \cdot Dv=0, &\ \ \ \ \ \text{ on } \ Q_r^* \\
v=u-R_{k_0}, &\ \ \ \ \ \text{ on } \ \partial_p Q_r^+ \setminus Q_r^*
\end{cases}
\end{align*}
where $G(M)=F(M+D_{k_0})-\Gamma_{k_0}$ which is an elliptic operator with the same ellipticity constants as $F$. Note that $G(O)=F(D_{k_0})-\Gamma_{k_0}=0$. Also by $D\beta^0$ we denote the matrix $(D\beta^0)_{ij}=(\beta_i)_{x_j}(0,0)$, $i=1,\dots ,n, \ j=1,\dots , n-1$. Then $v$ satisfies ABPT-estimate for the oblique derivative case (see Theorem \ref{ABP-Oblique}) which gives
\begin{equation} \label{oblique_general_H2+a_est6} 
\osc_{Q_{r}^+} v \leq \osc_{Q_{r}^+}\left( u(X,t) - R_{k_0}(X,t) \right).
\end{equation}
From Lemma \ref{oblique_general_H2+a_ii} we have that  $\bar{B}:=Dv(0,0)$, $\bar{\Gamma}:=v_t(0,0)$, $\bar{D}:=D^2v(0,0)$ exist and for $\bar{R}(X,t)=\bar{B}\cdot X + \bar{\Gamma}t+\frac{1}{2}X^\tau \bar{D}X$ we have
\begin{equation} \label{oblique_general_H2+a_est7} 
\osc_{Q_{\tilde{r}}^+} \left(v(X,t)-\bar{R} (X,t) \right) \leq C_0 \left( \frac{\tilde{r}}{r}\right)^{2+\alpha_1} \ \osc_{Q_{r}^+} v 
\end{equation}
for any $\tilde{r} \leq \frac{r}{4}$ and also 
\begin{equation} \label{oblique_general_H2+a_est8} 
|\bar{B}| \leq  \frac{C}{r}\osc_{Q_{r}^+} v, \ |\bar{\Gamma}| \leq  \frac{C}{r^2}\osc_{Q_{r}^+} v, \ ||\bar{D}||_\infty \leq  \frac{C}{r^2}\osc_{Q_{r}^+} v.  
\end{equation}
Note that $(\beta^0+D\beta^0 \ 0) \cdot\bar{B}=0$ that is $\beta^0 \cdot \bar{B}=0$ and $F(\bar{D}+D_{k_0})-\Gamma_{k_0}-\bar{\Gamma}=0$. Also, $(\beta^0+D\beta^0x) \cdot Dv=0$ holds in the classical sense on $Q_r^*$ and we can differentiate this condition with respect to $x_i$, for any $i\leq n-1$ to get at $x=0$, $\sum_{j=1}^n \left[ \bar{D}_{ij} \beta_j^0+(\beta_j)_{x_i}^0\bar{B}_j\right]=0$. 

Next, we take $\tilde{r}=\eta r$ (for $0<\eta<\rho$) in (\ref{oblique_general_H2+a_est7} ). Hence
\begin{equation} \label{oblique_general_H2+a_est9} 
\osc_{Q_{\eta r}^+} \left(v(X,t)- \bar{R} (X,t) \right) \leq C_0 \eta^{2+\alpha_1} \osc_{Q_{r}^+} v.
\end{equation}
Now take (universal) $0<\eta<<1$ sufficiently small in order to have that $8 \ C_0 \eta^{\alpha_1}<1$. We denote by $1-\theta :=8\ C_0 \eta^{\alpha_1}$, where $0<\theta<1$ is a universal constant. Then 
\begin{equation} \label{oblique_general_H2+a_est10} 
\osc_{Q_{\eta r}^+} \left(v(X,t)-\bar{R}(X,t)\right) \leq \frac{(1-\theta)}{2} \ \eta^2 \ \frac{\bar{C}}{\rho^{2+\alpha}}Kr^{2+\alpha}.
\end{equation}

Now to return to $u$ we define $w=u-R_{k_0}-v$. Note that $F(D^2(R_{k_0}+v))-(R_{k_0}+v)_t=F(D_{k_0}+D^2v)-\Gamma_{k_0}-v_t=0$. Moreover we can easily check that $DR_{k_0} = D_{k_0}X+B_{k_0}$. That is, $w$ satisfies
\begin{align*}
\begin{cases}
w \in S_p\left( \frac{\lambda}{n}, \Lambda ,f \right), &\ \ \ \ \ \text{ in } \ Q_r^+ \\
\beta \cdot Dw =g- \beta \cdot \left( D_{k_0}X+B_{k_0} + Dv \right), &\ \ \ \ \ \text{ on } \ Q_r^* \\
w=0, &\ \ \ \ \ \text{ on } \ \partial_p Q_r^+ \setminus Q_r^*.
\end{cases}
\end{align*}

Now for $0<\mu <1$ (to be chosen universal) we denote by $\bar{r}:= r(1-\mu )<r$. We apply again Theorem \ref{ABP-Oblique}
\begin{align}\label{oblique_general_H2+a_est11}
\osc_{Q_{\bar{r}}^+} w &\leq Cr ||f||_{L^{n+1}(Q^+_r)}+ Cr ||g||_{L^\infty(Q^*_r)}+Cr ||\beta \cdot ( D_{k_0}X+B_{k_0})||_{L^\infty(Q^*_r)} \nonumber \\
&\ \ \ +Cr ||\beta \cdot Dv||_{L^\infty(Q^*_{\bar{r}})}+ \osc_{{\partial_p Q_{\bar{r}}^+ \setminus Q_{\bar{r}}^*}} w \nonumber \\
&=: \textbf{I}+\textbf{II}+\textbf{III}+\textbf{IV}+\textbf{V}.
\end{align}
We want to bound every term \textbf{I} - \textbf{V} by a term of order $r^{2+\alpha}$. We start with term \textbf{I}. We have 
$\textbf{I}\leq  Cr||f||_{L^\infty \left(Q_r^+ \right)} \ C(n) r^{\frac{n+2}{n+1}} \leq Cr^2||f||_{L^\infty \left(Q_r^+ \right)}$ then using the $H^\gamma$ regularity of $f$ and the fact that $f(0,0)=0$ we get
$$\textbf{I} \leq Cr^2||f-f(0,0)||_{L^\infty \left(Q_r^+ \right)} \leq CKr^{2+\gamma}.$$
Next, for term \textbf{II}, we use the $H^{1+\gamma}$-regularity of $g$ and the fact that $g(0,0)=0, \ Dg(0,0,)=0$, 
$$\textbf{II}  = Cr||g-g(0,0)-Dg(0,0)\cdot x||_{L^\infty(Q^*_r)} \leq Crr^{1+\gamma} K \leq CKr^{2+\gamma}.$$
We continue with term \textbf{III} and we study first the term
$$\textbf{A}:=(\beta^0 + D\beta^0x) \cdot ( D_{k_0}X+B_{k_0})=\beta^0 \cdot  D_{k_0}X+ D\beta^0x \cdot  D_{k_0}X+ D\beta^0x \cdot B_{k_0}$$
and 
$$\beta^0 \cdot  D_{k_0}X= \sum_{i=1}^n\beta_i^0\sum_{k=1}^{n-1} (D_{k_0})_{ik}x_k,\ \ D\beta^0x \cdot B_{k_0}= \sum_{i=1}^n\sum_{k=1}^{n-1} (\beta_i)_{x_k}^0x_k(B_{k_0})_i.$$
Hence, $\textbf{A}=D\beta^0x \cdot  D_{k_0}X$. Returning to \textbf{III}, we have
\begin{align*}
\textbf{III} &\leq Cr ||\beta-\beta^0 - D\beta^0x ||_{L^\infty(Q^*_r)}||D_{k_0}X+B_{k_0}||_{L^\infty(Q^*_r)}\\
&\ \ \ + Cr ||(\beta^0 + D\beta^0x) \cdot ( D_{k_0}X+B_{k_0})||_{L^\infty(Q^*_r)}\\
&\leq Cr r^{1+\gamma}(||D_{k_0}||_{\infty}+|B_{k_0}|)+Crr^2||D\beta^0||_\infty ||D_{k_0}||_\infty.
\end{align*}
Note also that $|B_{k_0}|\leq CK$ and $||D_{k_0}||_{\infty} \leq CK$ which can be derived by (\ref{oblique_general_H2+a_est3}) and the fact that $B_0=0$ and $||D_{k_0}||\leq CK$. Then $\textbf{III}\leq CKr^{2+\gamma}.$ Next for term \textbf{IV}, we use again the $H^{1+\gamma}$-regularity of $\beta$ and the fact that $(\beta^0+D\beta^0x) \cdot Dv=0$ on $Q_r^*$, we have
$$\textbf{IV} \leq Cr||\beta-\beta^0 - D\beta^0x ||_{L^\infty(Q^*_r)}||Dv||_{L^\infty(Q^*_{\bar{r}})} \leq C_2\rho^{1+\gamma}\frac{\bar{C}}{\rho^{2+\alpha}}Kr^{2+\alpha}.$$
Finally we examine term \textbf{V}. Let $(X_0,t_0) \in \partial_p Q_{\bar{r}}^+ \setminus Q_{\bar{r}}^*$. If $|X_0|=\bar{r}$ we choose $\bar{X}_0 \in \left(\partial B_r\right)^+$ so that $|X_0-\bar{X}_0|=\mu r \leq \sqrt{2 \mu}r$ and $\bar{t}_0=t_0$. If $|X_0|<\bar{r}$ then $t_0=-(1-\mu)^2r^2$ and we choose $\bar{t}_0=-r^2$ then $|t_0-\bar{t}_0|^{1/2}=r \sqrt{\mu (2- \mu)} \leq \sqrt{2 \mu}r$ and $\bar{X}_0=X_0$. In any case  $|X_0-\bar{X}_0|+|t_0-\bar{t}_0|^{1/2}\leq  \sqrt{2 \mu}r $ and $(\bar{X}_0,\bar{t}_0) \in \partial_p Q_{r}^+ \setminus Q_{r}^*$ that is $w \left( \bar{X}_0, \bar{t}_0 \right) =0$. Then
\begin{align}\label{oblique_general_H2+a_est11_1}
|w \left( X_0,t_0 \right)|\leq |\left( u \left( X_0,t_0\right) - R_{k_0}(X_0,t_0) \right)-\left( u \left( \bar{X}_0, \bar{t}_0 \right) -R_{k_0}(\bar{X}_0,\bar{t_0}) \right)|  +|v \left( X_0,t_0\right)-v\left( \bar{X}_0, \bar{t}_0 \right)|
\end{align}
and we bound these terms using $H^\alpha$-estimates. Indeed, we have that
\begin{align*}
\begin{cases}
F(D^2(u-R_{k_0})+D_{k_0}) -\Gamma_{k_0}-(u-R_{k_0})_t = f, &\ \ \ \ \ \text{ in } \ Q_{2r}^+ \\
\beta \cdot D(u-R_{k_0}) = g-\beta \cdot (D_{k_0}X+B_{k_0}), &\ \ \ \ \ \text{ on } \ Q_{2r}^*. \\
\end{cases}
\end{align*}
Then Theorem \ref{bdary_C^a_oblique} gives
\begin{align*}
||u-R_{k_0}||_{H^{\alpha_2} \left( \overline{Q}^+_{r} \right) } \leq &\ \frac{C}{r^{\alpha_2}} \ \left( ||u-R_{k_0}||_{L^\infty \left( Q_{2r}^+ \right)}+r^{\frac{n}{n+1}}||f||_{L^{n+1}\left( Q_{2r}^+ \right)} \right)\\
&+\frac{C}{r^{\alpha_2}} \left( r||g||_{L^\infty \left( Q_{2r}^* \right)} +r||\beta \cdot (D_{k_0}X+B_{k_0})||_{L^\infty \left( Q_{2r}^* \right)}\right).
\end{align*}
Next we apply to $v$ global $H^\alpha$-estimates. Note that the values of $v$ on the parabolic boundary equal to $u-R_{k_0}$ which is $H^{\alpha_2}$. So, for $0<\alpha_3 < <\alpha_2$ universal, we have
\begin{align*}
||v||_{H^{\alpha_3} \left( \overline{Q}^+_{r} \right) } &\leq \ \frac{C}{r^{\alpha_3}} \ \left( ||v||_{L^\infty \left( Q_{r}^+ \right)}+r^{\alpha_2}||u-R_{k_0}||_{H^{\alpha_2} \left( \overline{Q}^+_{r} \right) } \right)\\
&\leq \ \frac{C}{r^{\alpha_3}} \ \ ||u-R_{k_0}||_{L^\infty \left( Q_{2r}^+ \right)} \\
&\ \ \ +\frac{C}{r^{\alpha_3}} \left( r^{\frac{n}{n+1}}||f||_{L^{n+1}\left( Q_{2r}^+ \right)}+r||g||_{L^\infty \left( Q_{2r}^* \right)} +r||\beta \cdot (D_{k_0}X+B_{k_0})||_{L^\infty \left( Q_{2r}^* \right)}\right).
\end{align*}
Now, we return to (\ref{oblique_general_H2+a_est11_1}).
\begin{align*}
|w(X_0,t_0)| &\leq C \mu ^{\alpha_3/2}\ \ ||u-R_{k_0}||_{L^\infty \left( Q_{2r}^+ \right)} \\
&\ \ \ + C \mu ^{\alpha_3/2} \left( r^{\frac{n}{n+1}}||f||_{L^{n+1}\left( Q_{2r}^+ \right)}+r||g||_{L^\infty \left( Q_{2r}^* \right)} +r||\beta \cdot (D_{k_0}X+B_{k_0})||_{L^\infty \left( Q_{2r}^* \right)}\right) \\
&\leq \textbf{VI}+\textbf{I}'+\textbf{II}'+\textbf{III}'.
\end{align*}
For term \textbf{VI}, we use the hypothesis of the induction, $\textbf{VI} \leq C_1 \mu ^{\alpha_3/2}\frac{\bar{C}}{\rho^{2+\alpha}}Kr^{2+\alpha}$. Moreover for term \textbf{I}$'$, we have $\textbf{I}'\leq Cr^{\frac{n}{n+1}}||f||_{L^\infty \left(Q_r^+ \right)} \ C(n) r^{\frac{n+2}{n+1}} = Cr^2||f||_{L^\infty \left(Q_r^+ \right)}$ 
then using the $H^\gamma$ regularity of $f$ and the fact that $f(0,0)=0$ we get $\textbf{I}' \leq CKr^{2+\gamma}.$
Also, terms \textbf{II}$'$ and \textbf{III}$'$ are in fact the same as terms \textbf{II} and \textbf{III}. That is,
$$\textbf{V} \leq C_1 \mu ^{\alpha_3/2} \frac{\bar{C}}{\rho^{2+\alpha}}Kr^{2+\alpha}+ C \mu ^{\alpha_3/2}  r^{2+\gamma } \ K +C_2\rho^{1+\gamma}\frac{\bar{C}}{\rho^{2+\alpha}}Kr^{2+\alpha}.$$
So, returning to (\ref{oblique_general_H2+a_est11}), we have
$$\osc_{Q_{\bar{r}}^+} w \leq C K r^{2+\gamma }+C_1 \mu ^{\alpha_3/2} \frac{\bar{C}}{\rho^{2+\alpha}}Kr^{2+\alpha}+C_2\rho^{1+\gamma}\frac{\bar{C}}{\rho^{2+\alpha}}Kr^{2+\alpha}. $$

Next combining the above with (\ref{oblique_general_H2+a_est10}) and choosing $\mu < 1 - \eta$ (then $\eta<1-\mu$) we get
\begin{align}\label{oblique_general_H2+a_est12} 
&\osc_{Q_{ \eta r}^+} \left[ u(X,t)-(R_{k_0}+\bar{R})(X,t)  \right] \nonumber \\ &\leq \frac{1}{2}(1-\theta) \eta^2 \frac{\bar{C}}{\rho^{2+\alpha}}Kr^{2+\alpha}   +C K r^{2+\gamma }+C_1 \mu ^{\alpha_3/2} \frac{\bar{C}}{\rho^{2+\alpha}}Kr^{2+\alpha}+C_2\rho^{1+\gamma}\frac{\bar{C}}{\rho^{2+\alpha}}Kr^{2+\alpha}.
\end{align}
We choose the right constants $\alpha_0$, $\mu$ and $\bar{C}$. So, take $\alpha_0$ so that $\eta^{\alpha_0}=1-\frac{\theta}{2}$ and $\alpha=\min\{\alpha_0, \gamma\}$. Take $\mu \leq \frac{\eta^{\frac{2(2+\alpha)}{\alpha_3}}}{(4C_1)^{\frac{2}{\alpha_3}}}$, $\rho \leq \frac{\eta^{\frac{2+\alpha}{1+\gamma}}}{(4C_2)^{\frac{1}{1+\gamma}}}$ and $\bar{C}$ large enough so that $\frac{\eta \theta \bar{C}}{4\rho^{2+\alpha}} \geq  C$. Then we return to (\ref{oblique_general_H2+a_est12}) writing $1-\theta$ as $1-\frac{\theta}{2}-\frac{\theta}{2}$ and recalling that $r=\frac{\rho \eta^{k_0}}{2}\leq \eta^{k_0}$,
\begin{align*}
\osc_{Q_{\rho \eta^{k_0+1}}^+} \left[ u(X,t)-(R_{k_0}+\bar{R})(X,t) \right] &\leq K \left[\frac{1}{2} \left(1-\frac{\theta}{2}\right)\bar{C}\eta^2 \eta^{k_0(2+\alpha)} +C r^{2+\alpha } +\bar{C}\frac{\eta^{2+\alpha}}{2} \eta^{k_0(2+\alpha)}\right]\\
& \ \ \ -K\frac{\eta \theta \bar{C}}{4\rho^{2+\alpha}} r^{2+\alpha} \leq \bar{C}K \eta^{(k_0+1)(2+\alpha)}.
\end{align*}
For $R_{k_0+1}=R_{k_0}+\bar{R}$ we have (\ref{oblique_general_H2+a_est2}) for $k_0+1$. Note also that $F(D_{k_0}+\bar{D})-(\Gamma_{k_0}+\bar{\Gamma})=0$, $\beta^0 \cdot B_{k_0+1} =0$ and for any $i=1,\dots n-1$, $\sum_{j=1}^n \left[ (D_{k_0+1})_{ij}\beta_j^0+(\beta_j)_{x_i}^0(B_{k_0+1})_j\right]=0$. It remains to get (\ref{oblique_general_H2+a_est3}) for $k=k_0$. To do so, we use relation (\ref{oblique_general_H2+a_est8}) together with (\ref{oblique_general_H2+a_est6}) and then (\ref{oblique_general_H2+a_est4}).

Then we can finish the proof in the same way as in the proof of Lemma \ref{oblique_general_H2+a_i}.
\end{proof}

\appendix

\section{Auxiliary Results}

In this section  we provide the proofs of results mentioned in the text for completeness (see \cite{Wang2}). We start with the proof of Lemma \ref{osc_u/y}. The following Lipschitz-estimate is used. It can be proved using a barrier argument, see for instance Lemma 2.1 in \cite{BCN}.

\begin{prop} \label{Lipschitz}
Let $f$ be bounded in $Q_1^+$ and $u \in C \left( Q_1^+ \cup Q_1^*  \right)$ be bounded and satisfy in the viscosity sense
\begin{align*}
\begin{cases}
u \in S_p(\lambda,\Lambda,f), &\ \ \ \ \ \text{ in } \ Q_1^+ \\
u = 0, &\ \ \ \ \ \text{ on } \ Q_1^*. \\
\end{cases}
\end{align*}
Then there exists universal constant $C>0$ so that 
\begin{equation}\label{Lips_est}
|u(X,t)| \leq C \left( ||u||_{L^\infty \left( Q^+_1\right)} +||f||_{L^{n+1}\left( Q^+_1 \right)} \right) \ y
\end{equation}
for every $(X,t)=(x,y,t) \in \overline{Q}^+_{1/2}$.
\end{prop}

\begin{proof}[Proof of Lemma \ref{osc_u/y}]
The idea of the proof of Lemma \ref{osc_u/y} is based on the proof of Theorem 9.31 in \cite{GT} or on its parabolic version appeared in \cite{Lieb} (Lemma 7.46 
and 7.47).

First we observe that $\frac{u}{y}$ is bounded in $\overline{Q}_{1/2}^+$ from Proposition \ref{Lipschitz}. It is enough to show
\begin{equation} \label{osc_u/y_est1}
\osc_{\overline{Q}_{\tau \rho}^+} \frac{u}{y} \leq \gamma \osc_{\overline{Q}_\rho^+} \frac{u}{y}+ C||f||_{L^\infty \left( Q_1^+ \right)} , \ \text{ for every } \ 0<\rho \leq \frac{1}{2}
\end{equation}
where $0<\tau, \gamma<1$ and $C>0$ are universal constants, then (\ref{osc_u/y_est}) follows by standard iteration. To get (\ref{osc_u/y_est1}) we use a barrier argument in order to be able to apply Harnack inequality to $\frac{u}{y}$ up to the flat boundary.

First we consider the case when $u \geq 0$ in $Q_1^+$.
\\ \underline{Step 1.} Set $v := \frac{u}{y}$. Then for any $0<\rho \leq \frac{1}{2}$, $0<\delta \leq 1$ and $A=(0,\dots,0,\rho)$ we see that $Q_{\rho/2}(A,0) \subset H(\rho,1)$ and we apply Harnack inequality there. For $K_R := B_{\frac{R^2}{2 \sqrt{2} }}(0,0) \times \left[ -R^2 + \frac{3}{8} R^4, -R^2+ \frac{4}{8} R^4 \right]$, where $0<R <<1$ universal constant,
\begin{align*}
\sup_{K_{\frac{\rho R}{2}}(A,0)} v  \leq \frac{2}{\rho}  \sup_{K_{\frac{\rho R}{2}}(A,0)} u &\leq \frac{2}{\rho} C \left( \inf_{Q_{\frac{\rho R^2}{2}}(A,0)} u +\rho \ ||f||_{L^\infty \left( Q_1^+ \right)} \right) \\
&\leq \frac{2}{\rho} C \left(  \frac{3\rho}{2}\inf_{Q_{\frac{\rho R^2}{2}}(A,0)} v +\rho \ ||f||_{L^\infty \left( Q_1^+ \right)} \right).
\end{align*}
Hence, defining the following thin set,
$$H'(\rho,\delta) := \left\lbrace (X,t): |x|<\frac{\rho R^2}{4}, y=\delta \rho, -\frac{\rho^2 R^2}{16}<t\leq 0 \right\rbrace$$
which lies in $Q_{\frac{\rho R^2}{2}}(A,0)$ for $0<\delta< \frac{\sqrt{3}R^2}{4}$, we have $\ 
\sup_{K_{\frac{\rho R}{2}}(A,0)} v  \leq C \left(  \inf_{H'(\rho,\delta)} v + \ ||f||_{L^\infty \left( Q_1^+ \right)} \right)$.

\underline{Step 2.} Now using a suitable barrier argument we will get an estimate up to the flat boundary, $\
\inf_{H'(\rho,\delta)} v \leq C \left( \inf_{\tilde{H}(\frac{\rho}{4},\delta)} v  +  ||f||_{L^\infty \left( Q_1^+ \right)} \right)$, where 
$$\tilde{H} (\rho,\delta) := \left\lbrace (X,t): |x|<\frac{\rho R^2}{4}, 0<y<\delta \rho, -\frac{\rho^2 R^2}{16}<t\leq 0 \right\rbrace.$$

For convenience we consider the function $\bar{u} := \frac{1}{m} u$, where $m:=\inf_{H'(\rho,\delta)} v$. Then $\bar{u} \in S_p (\lambda, \Lambda, \bar{f})$ in $Q_1^+$, where $\bar{f}:= \frac{f}{m}$. Moreover, if we denote by $\bar{v}:= \frac{\bar{u}}{y}$ then we want to get 
$$C \left( \inf_{\tilde{H}(\frac{\rho}{4},\delta)} \bar{v}  + ||\bar{f}||_{L^\infty \left( Q_1^+ \right)} \right) \geq 1.$$ 
For, we define
$$b(X,t)=y \ \left[ 1 - \frac{|x|^2}{\tilde{\rho}^2} + \frac{t}{\tilde{\rho}^2} + \left( \frac{1+||\bar{f}||_{L^\infty \left( Q_1^+ \right)}}{\lambda} \right) \left( \frac{y-\delta \rho}{\sqrt{\delta}\rho} \right) \right] \ \text{ for } \ (X,t) \in \tilde{H} (\rho,\delta)$$
where $\tilde{\rho}:=\frac{\rho R^2}{4}$. Our intention is to apply a comparison principle for $b$ and $\bar{u}$. We show 
\begin{enumerate}
\item $M^- (D^2b)-b_t \geq \bar{f}$ in $\tilde{H} (\rho,\delta)$. Then $\bar{u}-b \in \overline{S}_p(\lambda,\Lambda,0)$ in $\tilde{H} (\rho,\delta)$.
\item $\bar{u} - b \geq 0$ on $\partial_p\tilde{H} (\rho,\delta)$.
\end{enumerate}

Recall that $\mathcal{M}^- (M,\lambda,\Lambda )  = \inf_{A \in \mathcal{A}_{\lambda, \Lambda }} L_A(M)$, where $\mathcal{A}_{\lambda, \Lambda }$ be the subset of $S_n$ containing all matrices whose eigenvalues lie in the interval $[\lambda,\Lambda]$ and for $A \in \mathcal{A}_{\lambda, \Lambda }$, $L_A$ is the linear functional $L_A(M)=tr (AM)$, where $M \in S_n$. So we want to show that, for any such linear operator $L_A$, $L_A(D^2b)-b_t \geq \bar{f}$. Take any  $A \in \mathcal{A}_{\lambda, \Lambda }$ and observe that $\lambda \leq  a_{ii} \leq \Lambda $ and $|a_{in}|\leq \Lambda - \frac{\lambda}{2}=:C_0>0$. So in $\tilde{H} (\rho,\delta)$, using that $y<\rho \delta, \ |x|< \tilde{\rho}, \ \delta < \sqrt{\delta}$, we compute
\begin{align*}
L_A(D^2b)-b_t \geq - \frac{16}{\rho R^4}(1+2n\Lambda) \sqrt{\delta} - \frac{16 C_0n}{\rho R^2}+ \frac{2}{\sqrt{\delta}\rho}\left( 1+||\bar{f}||_{L^\infty \left( Q_1^+ \right)}\right).
\end{align*}
That is, it is enough to show that $ - \frac{16}{\rho R^4}(1+2n\Lambda) \delta - \frac{16 C_0n}{\rho R^2} \sqrt{\delta}+ 2\geq 0.$
The above is a polynomial in $\bar{\delta}:= \sqrt{\delta}$. One can observe that this polynomial has two universal roots $\ \bar{\delta}_1<0,\ \bar{\delta}_2>0$ and the polynomial is positive in $(\bar{\delta}_1,\bar{\delta}_2)$. So if we choose $0<\delta<\bar{\delta}_1^2$ we have the desired. 

Now we examine $b$ on $\partial_p\tilde{H} (\rho,\delta)$. We split the boundary data in the following cases
\begin{enumerate}
\item[$\bullet$] For $y=0$, $b=0=u=\bar{u}$.
\item[$\bullet$] For $y=\delta \rho$, $b(x,\delta \rho,t) = \delta \rho \ \left( 1 - \frac{|x|^2}{\tilde{\rho}^2} + \frac{t}{\tilde{\rho}^2} \right) \leq \delta \rho \leq \bar{u}(x,\delta \rho,t)$.
\item[$\bullet$] For $t=-\tilde{\rho}^2$, $b(X,-\tilde{\rho}^2)=y \ \left[  - \frac{|x|^2}{\tilde{\rho}^2}  + \left( \frac{1+||\bar{f}||_{L^\infty \left( Q_1^+ \right)}}{\lambda} \right) \left( \frac{y-\delta \rho}{\sqrt{\delta}\rho} \right) \right] \leq 0 \leq \bar{u}(X,-\tilde{\rho}^2).$
\item[$\bullet$] For $|x|=\tilde{\rho}$, $b(X,t)=y \ \left[   \frac{t}{\tilde{\rho}^2}  + \left( \frac{1+||\bar{f}||_{L^\infty \left( Q_1^+ \right)}}{\lambda} \right) \left( \frac{y-\delta \rho}{\sqrt{\delta}\rho} \right) \right] \leq 0 \leq \bar{u}(X,t).$
\end{enumerate}

Therefore $\bar{u}-b \geq 0$ in $\tilde{H} (\rho,\delta)$ and as a consequence, in $\tilde{H} \left( \frac{\rho}{4},\delta \right)$ we have an estimate by below for the ratio
\begin{align*}
\frac{\bar{u}(X,t)}{y} \geq\frac{7}{8} - \frac{\sqrt{\delta}}{\lambda} - \frac{||\bar{f}||_{L^\infty \left( Q_1^+ \right)}}{\lambda} \geq \frac{1}{2} - \frac{||\bar{f}||_{L^\infty \left( Q_1^+ \right)}}{\lambda}
\end{align*}
using $|x| < \frac{\tilde{\rho}}{4}, t > -\frac{\tilde{\rho}^2}{16},\ y >0$ and choosing $1 \leq \delta \leq \left( \frac{3\lambda}{8} \right)^2$. Hence taking infimum  we get the desired

Next we remove the assumption on the nonnegativity of $u$. 
\\ \underline{Step 3.} We denote $M:= \sup_{\tilde{H} \left( 2\rho,\delta \right)}v$ and $m:=\sup_{\tilde{H} \left( 2\rho,\delta \right)}v$. Then the functions $My-u, u-my$ are nonnegative. Applying Step 2 to these two functions and then adding the two estimates we conclude
$$\osc_{\tilde{H} \left( \frac{\rho}{4},\delta \right)} v \leq \frac{C-1}{C} \osc_{\tilde{H} \left( 2\rho,\delta \right)}v + 2C \  ||\bar{f}||_{L^\infty \left( Q_1^+ \right)}.$$
\end{proof}

Then we examine the $H^{1+\alpha}$ regularity for the nonlinear parabolic Dirichlet problem (Theorem \ref{Dirichlet_H1+a}). We start by studying the homogeneous case using Lemma \ref{osc_u/y}.

\begin{lemma} \label{Dirichlet_H1+a_hom} 
Let $u \in C \left( Q_r^+ \cup Q_r^*  \right)$ be bounded and satisfy in the viscosity sense
\begin{align} \label{Dir_prob_hom}
\begin{cases}
F(D^2u) -u_t = 0, &\ \ \ \ \ \text{ in } \ Q_r^+ \\
u = 0, &\ \ \ \ \ \text{ on } \ Q_r^*. \\
\end{cases}
\end{align}
Then the first derivatives $u_{x_1},\dots,u_{x_{n-1}},u_y$ exist in $\overline{Q}_{r/2}^+$. Moreover there exists universal constant $0<\alpha < 1$ so that $u$ is punctually $H^{1+\alpha}$ at every point $P_0 \in Q_{r/2}^*$. More precisely for $b_{P_0}=u_y(P_0)$ and any $\tilde{r} \leq \frac{r}{2}$
\begin{equation}\label{Dirichlet_H1+a_est_hom}
|u(X,t) -b_{P_0}y| \leq C \ \frac{\tilde{r}^{1+\alpha}}{r^{1+\alpha}} \left( ||u||_{L^\infty \left( Q_{r}^+ \right)} +r^2|F(O)| \right) 
\end{equation}
for every $(X,t) \in \overline{Q}_{\tilde{r}}^+(P_0)$, where $C>0$ is a universal constant.
\end{lemma}

\begin{proof}
First let us examine what Lemma \ref{osc_u/y} ensures:
\begin{enumerate}
\item[$\bullet$] $u_y$ exists on $Q_r^*$. Indeed we show this at $(0,0)$. Let the sequence $\{h_k\}_k$ be so that $h_k \searrow 0$ as $k \to \infty$
and take $m>l$ (large enough) then applying Lemma \ref{osc_u/y} (rescaled) we obtain
$$ \frac{u(0,h_m,0)}{h_m} - \frac{u(0,h_l,0)}{h_l} \leq \frac{C}{r^\alpha}K(h_l)^\alpha$$
where $K:= \osc_{\overline{Q}_{r/2}^+} \frac{u}{y} + |F(O)|$. That is the sequence $ \{ \frac{u(0,h_k,0)}{h_k} \}$ is a Cauchy sequence and hence it converges to $u_y(0,0)$ (since $u(0,0)=0$). 
\item[$\bullet$] $u_y \in H^\alpha \left(Q_{r/2}^* \right)$.Indeed, let $h< \frac{\rho}{2}, \rho<\frac{r}{2}$ and $(x_0,t_0),(z_0,s_0) \in Q_{\rho/2}^*$ then
$$ \frac{u(x_0,h,t_0)}{h} - \frac{u(z_0,h,s_0)}{h} \leq \frac{C}{r^\alpha}K\rho^\alpha.$$
Taking $h \to 0$ we obtain $\ \osc_{Q_{\rho/2}^*} u_y \leq \frac{C}{r^\alpha}K\rho^\alpha$.
\end{enumerate}

Now let $(X,t) \in \overline{Q}_{\tilde{r}}^+$ and $h>0$ small, 
$$\frac{u(X,t)}{y} - \frac{u(0,h,0)}{h} \leq  C \left( \frac{\tilde{r}}{r}\right)^\alpha \left( \frac{1}{r} \osc_{\overline{Q}_{r}^+} u+r^2 |F(O)| \right).$$
Then letting $h \to 0^+$ and since $0<y \leq \tilde{r}$ we get
$$|u(X,t) - u_y(0,0)y| \leq  \  C \left( \frac{\tilde{r}}{r}\right)^{1+\alpha} \left( ||u||_{L^\infty\left(Q_r^+\right)} +r^2 |F(O)| \right).$$
\end{proof}

Next we go from the homogeneous to the non-homogeneous case using the standard approximating procedure used also in Theorems \ref{oblique_general_H1+a}, \ref{oblique_general_H2+a} and \ref{oblique_general_H2+a_i}. We give the proof briefly for completeness.

\begin{proof}[Proof of Theorem \ref{Dirichlet_H1+a}]
We will show the theorem around $P_0=(0,0)$. Note that without the loss of generality we can assume that $u(0,0)=g(0,0)=0$ and $\nabla_{n-1}g(0,0)=0$ (since we can consider the transformation $u(X,t)-g(0,0)-\nabla_{n-1}g(0,0) \cdot x$). For convenience  let us denote $K:=||u||_{L^\infty \left( Q_{1}^+ \right)}+||g||_{H^{1+\alpha}\left(  \overline{Q}_{1/2}^* \right)}  +|F(O)|$.

We intend to find a number $A \in \R$ so that, for universal $C>0, 0<\gamma<1, \alpha_0>0$ and $\beta = \min\{\alpha,\alpha_0\}$, we will have
\begin{equation} \label{Dirichlet_H1+a_est1}
\osc_{Q_{\gamma^k}^+} \left( u(X,t) - Ay \right) \leq CK \gamma^{k(1+\beta)}, \ \ \text{ for any } \ \ k \in \N. 
\end{equation} 

Now, to prove (\ref{Dirichlet_H1+a_est1}) we are going to show by induction that there exist universal constants $0<\gamma <<1, \bar{C}>0, \alpha_0>0$ such that for $\beta :=\min \{\alpha, \alpha_0\}$ we can find a number $A_k \in \R$ for any $k \in \N$ so that
\begin{equation} \label{Dirichlet_H1+a_est2}
\osc_{Q_{\gamma^k}^+} \left( u(X,t) - A_ky \right) \leq \bar{C}K \gamma^{k(1+\beta)}
\end{equation}
and
\begin{equation} \label{Dirichlet_H1+a_est3}
|A_{k+1}-A_k| \leq CK\gamma^{k\beta}.
\end{equation}
Note that the right constants will be deduced from the induction. The details follow.

First, for $k=0$, take $A_0=0$ and choose any $\bar{C}\geq 2$. Next for the induction we assume that we have found numbers $A_0,\dots,A_N$ for which (\ref{Dirichlet_H1+a_est2}) and (\ref{Dirichlet_H1+a_est3}) are true. 

Now we consider a suitable problem with homogeneous Dirichlet data on the flat boundary in order to use Theorem \ref{Dirichlet_H1+a_hom}. Let $v$ be the viscosity solution of
\begin{align*}
\begin{cases}
F(D^2v) -v_t = 0, &\ \ \ \ \ \text{ in } \ Q_r^+ \\
v=0, &\ \ \ \ \ \text{ on } \ Q_r^* \\
v=u-By, &\ \ \ \ \ \text{ on } \ \partial_p Q_r^+ \setminus Q_r^*.
\end{cases}
\end{align*}
Then $v$ satisfies maximum principle which gives
\begin{equation} \label{Dirichlet_H1+a_est6}
\osc_{Q_{r}^+} v \leq \osc_{Q_{r}^+}\left( u(X,t) - By \right) +Cr^2|F(O)|.
\end{equation}
From Lemma \ref{Dirichlet_H1+a_hom} we have that $A:=v_y(0,0)$ exists and
\begin{equation} \label{Dirichlet_H1+a_est7}
\osc_{Q_{\tilde{r}}^+} \left(v(X,t)-Ay \right) \leq C_0 \left( \frac{\tilde{r}}{r}\right)^{1+\alpha_1} \left(\osc_{Q_{r}^+} v + r^2|F(O)|\right)
\end{equation}
for any $\tilde{r} \leq \frac{r}{2}$ and also $\ 
|A| \leq C \left( \frac{1}{r}\osc_{Q_{r}^+} v + r^2|F(O)|\right).$

Next, we take $\tilde{r}=\gamma r$ (note that $\gamma$ is very small) in (\ref{Dirichlet_H1+a_est7}). Hence
\begin{equation} \label{Dirichlet_H1+a_est9}
\osc_{Q_{\gamma r}^+} \left(v(X,t)-Ay \right) \leq C_0 \gamma^{1+\alpha_1} \osc_{Q_{r}^+} v + C_0r^2\gamma|F(O)|
\end{equation}
since $\gamma^{1+\alpha_1} \leq \gamma$. Now take (universal) $\gamma<<1$ sufficiently small in order to have that $C_0 \gamma^{\alpha_1}<1$. We denote by $1-\theta :=C_0 \gamma^{\alpha_1}$, where $0<\theta<1$ is a universal constant. Then combining (\ref{Dirichlet_H1+a_est9}) and (\ref{Dirichlet_H1+a_est6}) we obtain
\begin{equation} \label{Dirichlet_H1+a_est10}
\osc_{Q_{\gamma r}^+} \left(v(X,t)-Ay \right) \leq (1-\theta) \gamma \osc_{Q_{r}^+}\left( u(X,t) - By \right) +Cr^2|F(O)|.
\end{equation}

Now to return to $u$ we define $w=u-By-v$. Then
\begin{align*}
\begin{cases}
w \in S_p\left( \frac{\lambda}{n}, \Lambda \right), &\ \ \ \ \ \text{ in } \ Q_r^+ \\
w=g, &\ \ \ \ \ \text{ on } \ Q_r^* \\
w=0, &\ \ \ \ \ \text{ on } \ \partial_p Q_r^+ \setminus Q_r^*.
\end{cases}
\end{align*}

Subsequently, applying again maximum principle we obtain
$\osc_{Q_{ r}^+} w \leq C ||g||_{L^\infty(Q^*_r)}.$
The regularity we have assumed for $g$ will give the right decay for the oscillation of $w$. That is, (since $g(0,0)=0, \nabla_{n-1}g(0,0)=0$)
\begin{align*}
|g(x,t)| &= |g(x,t)-g(0,0)-\nabla_{n-1}g(0,0) \cdot x| \ \ \ \text{ for } \ (x,t) \in Q_r^* \\
&\leq |g(x,t)-g(0,t)-\nabla_{n-1}g(0,t) \cdot x|  +|g(0,t)-g(0,0)|\\
&\ \ \ + |x| \ |\nabla_{n-1}g(0,t) -\nabla_{n-1}g(0,0)|\\
&\leq C ||g||_{H^{1+\alpha}\left(\overline{Q}_{1/2}^*\right)} \left( |x|^{1+\alpha}+|t|^{\frac{1+\alpha}{2}}+|t|^{\frac{\alpha}{2}}|x| \right)\\
&\leq C ||g||_{H^{1+\alpha}\left(\overline{Q}_{1/2}^*\right)} \left( \max\{|x|,|t|^{1/2} \} \right)^{1+\alpha} \leq C ||g||_{H^{1+\alpha}\left(\overline{Q}_{1/2}^*\right)}r^{1+\alpha}.
\end{align*}
Hence we obtain
\begin{equation} \label{Dirichlet_H1+a_est11}
\osc_{Q_{ r}^+} w \leq C r^{1+\alpha} ||g||_{H^{1+\alpha}\left(\overline{Q}_{1/2}^*\right)}.
\end{equation}

Adding (\ref{Dirichlet_H1+a_est10}) and (\ref{Dirichlet_H1+a_est11}) yields
$$\osc_{Q_{ \gamma r}^+} \left[ u(X,t)-(A+B)y \right] \leq (1-\theta) \gamma \osc_{Q_{r}^+}\left( u(X,t) - By \right) +Cr^2|F(O)|+C ||g||_{H^{1+\alpha}\left(\overline{Q}_{1/2}^*\right)}r^{1+\alpha}.$$
Recalling that $r=\gamma^N$ and using the hypotheses we get
\begin{equation}\label{Dirichlet_H1+a_est12}
\osc_{Q_{ \gamma^{N+1}}^+} \left[ u(X,t)-(A+B)y \right] \leq K \left[ (1-\theta)\bar{C}\gamma \gamma^{N(1+\beta)}+C \left( \gamma^{2N} + \gamma^{N(1+\alpha)} \right) \right].
\end{equation}
We have to choose the right constants $\alpha_0$ and $\bar{C}$.  Take $\alpha_0$ so that $\gamma^{\alpha_0}=1-\frac{\theta}{2}$ and $\bar{C}$ large enough so that $\frac{\gamma \theta \bar{C}}{4} \geq C$ (note that our choices are independent of $N$). Then we return to (\ref{Dirichlet_H1+a_est12}) writing $1-\theta$ as $1-\frac{\theta}{2}-\frac{\theta}{2}$ and recalling that $\beta = \min \{\alpha, \alpha_0 \}$,
\begin{align*}
\osc_{Q_{ \gamma^{N+1}}^+} \left[ u(X,t)-(A+B)y \right] &\leq K \left[ \left( 1-\frac{\theta}{2}\right)\bar{C}\gamma \gamma^{N(1+\beta)}+C \left( \gamma^{2N} + \gamma^{N(1+\alpha)} \right) - \frac{\theta}{2}\bar{C}\gamma \gamma^{N(1+\beta)}\right] \\
&\leq K \bar{C} \gamma^{(N+1)(1+\beta)}.
\end{align*}
Choosing $A_{N+1}=A_N+A$ the inductive proof is completed. 

Then the limit $\lim_{k\to \infty} A_k$ is the number $A$ of (\ref{Dirichlet_H1+a_est1}).
\end{proof}

Finally we prove a closedness result used in the text.

\begin{prop} (Closedness). \label{neumannclosedness}
Let $\{u_k \}_{k \in \mathbb{N}} \subset C(Q_1^+ \cup Q_1^*)$ are such that for every $k \in \mathbb{N}$, $u_k$ satisfies in the viscosity sense the following
\begin{align} \label{subproblem}
\begin{cases}
F \left( D^2 v(X,t) \right)-v_t(X,t)  \geq 0, & \ \ \ \ \ (X,t) \in Q_1^+ \\
v_y (x,0,t) \geq 0, & \ \ \ \ \ (x,t) \in Q_1^* \\
\end{cases}
\end{align}
Assume that $u_k$ converges to $u$ uniformly in any $\overline{ Q}_\rho ^+(x_0,0,t_0) \subset Q_1^+ \cup Q_1^*$, then $u$ satisfies (\ref{subproblem}) in the viscosity sense. 
\end{prop}

\begin{proof}
First note that proving that $F \left( D^2 u \right)-u_t  \geq 0$ in $Q_1^+$ in the viscosity sense is standard, see for example Proposition 2.9 in \cite{CC}. So, it remains to study the Neumann sub-condition and the proof is a suitable modification of the one for the equation. 

Take any point $P_0=(x_0,0,t_0) \in Q_1^*$ and any test function $\phi$ that touches $u$ by above at $P_0$ in $\overline{Q}^+_\rho(P_0) \in Q_1^+$. We want to show that, $\phi_y(P_0) \geq 0$.

We have for $\epsilon >0 $ and any $0<r<\rho$, $u(X,t) -\phi (X,t) - \frac{\epsilon}{2} (|X-x_0|^2 -t+t_0) < 0$, for $(X,t) \in Q_r^+(P_0) \setminus \{P_0\}.$ Denoting by $\tilde{\phi}(X,t) := \phi (X,t) +\frac{\epsilon}{2} (|X-X_0|^2 -t+t_0)$ and by $A_r(P_0):= \partial_p Q_r^+(P_0) \setminus Q_r^*(P_0)$  we consider, 
$$c:= \max_{(X,t) \in  A_r(P_0)} \left( u(X,t) - \tilde{\phi} (X,t) \right) <0.$$
Then $u- \tilde{\phi}  \leq c$ on $A_r(P_0)$.  

Using the uniform convergence of $u_k$ to $u$ and the definition of $c$, we have for large enough $k$, $u_k(X,t) - \tilde{\phi }(X,t) < u_k(P_0)- \tilde{\phi} (P_0) +\frac{c}{2}$, for any $(X,t) \in A_r(P_0)$. Set
$$C_k:= \max_{(X,t) \in  \overline{Q}_r^+(P_0)} (u_k(X,t)- \tilde{\phi} (X,t))$$ 
which is achieved at some point $(X_k,t_k) \in Q_r^+(P_0) \cup Q_r^*(P_0)$.

Therefore, for any large enough $m \in \N$ there exist points $(X_{k_m},t_{k_m}) \in Q_{1/m}^+(P_0) \cup Q_{1/m}^*(P_0)$ so that $(X_{k_m},t_{k_m}) \to P_0$, as $m \to \infty$ and the test function $ \psi_{k_m} := \tilde{\phi} +C_{k_m}$ touches by above $u_{k_m}$  at $(X_{k_m},t_{k_m})$. Hence, we treat two cases:
\begin{enumerate}
\item[\textbf{\textit{1.}}] If $(X_{k_m},t_{k_m}) \in Q_r^*(P_0)$ we have that $ \left( \psi_{k_m} \right)_y (X_{k_m},t_{k_m}) \geq 0$, hence $
\phi_y(X_{k_m},t_{k_m}) \geq 0$.
\item[\textbf{\textit{2.}}] If $(X_{k_m},t_{k_m}) \in Q_r^+(P_0)$ we have that $F(D^2 \phi (X_{k_m},t_{k_m}) + \epsilon I)-\phi_t(X_{k_m},t_{k_m}) +\frac{\epsilon}{2} \geq 0$.
\end{enumerate}

Now, if \textbf{\textit{1.}} is true for an infinite number of $m$'s then taking a suitable subsequence and passing to the limit we derive, $\phi_y(P_0) \geq 0$
as desired. Otherwise, \textbf{\textit{2.}} will be true for an infinite number of $m$ and so taking subsequences and limits we derive, $F(D^2 \phi (P_0))-\phi_t(P_0)) \geq 0.$

To finish the proof we assume that $\phi_y(P_0) < 0$ (to get a contadiction). Then having in mind the dichotomy above we conclude that $F(D^2 \phi (P_0))-\phi_t(P_0) \geq 0$ must be true. For small $\gamma>0$, we consider the perturbation of $\phi$, $\phi_\gamma (X,t) = \phi (X,t) + \gamma y - \frac{y^2}{\gamma}.$
Observe that if $(X,t) \in Q_{\gamma^2}^+ (P_0)$, then $\gamma y - \frac{y^2}{\gamma} \geq 0$. Therefore, we obtain that $\phi_\gamma$  touches $u$ by above at $P_0$  and following the same steps as we did for $\phi$ we conclude that
$$\left( \phi_\gamma \right) _y(P_0) \geq 0 \ \ \text{ or } \ \ F(D^2 \phi_\gamma (P_0))-\left( \phi_\gamma \right)_t(P_0) \geq 0.$$
A direct computation of these quantities and choosing $\gamma$ small enough (so that $\gamma<-\phi_y(P_0)$, $\frac{2 \lambda}{\gamma}> F \left( D^2 \phi (P_0) \right)- \phi_t(P_0) +1$).
\end{proof}

\bibliographystyle{plain}   
\bibliography{biblio}             
\index{Bibliography@\emph{Bibliography}}%

\vspace{3em}

\begin{tabular}{l}
Georgiana Chatzigeorgiou\\ University of Cyprus \\ Department of Mathematics \& Statistics \\ P.O. Box 20537\\
Nicosia, CY- 1678 CYPRUS
\\ {\small \tt chatzigeorgiou.georgiana@ucy.ac.cy}
\end{tabular}
\hspace{17mm}
\begin{tabular}{lr}
Emmanouil Milakis\\ University of Cyprus \\ Department of Mathematics \& Statistics \\ P.O. Box 20537\\
Nicosia, CY- 1678 CYPRUS
\\ {\small \tt emilakis@ucy.ac.cy}
\end{tabular}

%

\end{document}